\documentclass[a4paper,11pt]{article}
\usepackage{amsfonts}
\usepackage{amssymb}
\usepackage{amsmath}
\usepackage{mathtools}
\usepackage{mathrsfs}
\usepackage{bbm}
\usepackage{xcolor}
\usepackage[utf8]{inputenc}
\usepackage[margin=1.5in]{geometry}
\usepackage[T1]{fontenc}
\usepackage{amsthm}
\usepackage[english]{babel}
\usepackage{graphicx}
\usepackage{tikz-cd}
\usepackage{tikz}
\usepackage{dsfont}
\usetikzlibrary{matrix, arrows}

\newcommand{\nn}{\mathbb{N}}
\newcommand{\qq}{\mathbb{Q}}
\newcommand{\rr}{\mathbb{R}}
\newcommand{\cc}{\mathbb{C}}

\newcommand{\oo}{\mathcal{O}}

\newcommand{\mi}{^{-1}}

\newtheorem{theorem}{Theorem}[subsection]
\newtheorem*{theorem*}{Theorem}
\newtheorem{prop}[theorem]{Proposition}
\newtheorem{lemma}[theorem]{Lemma}
\newtheorem{corollary}[theorem]{Corollary}
\theoremstyle{definition}
\newtheorem{example}[theorem]{Example}
\theoremstyle{definition}
\newtheorem{remark}[theorem]{Remark}
\theoremstyle{definition}
\newtheorem{defi}[theorem]{Definition}

\DeclareMathOperator{\ord}{ord}

\DeclareMathOperator{\vol}{vol}

\DeclareMathOperator{\Specop}{Spec}

\newcommand\Spec[1]{\Specop \mathrm{#1}}

\newcommand\charac[1]{\mathrm{char}\, #1}

\newcommand\vp{\varphi}
\newcommand\nv{\mathcal{N}(V)}
\newcommand\cpe{\mathrm{(ce)}}

\newcommand\field{\mathbb{K}}
\newcommand\norm{\zeta}
\newcommand\bsni{\bigskip\noindent}
\newcommand\an{^{\mathrm{an}}}
\newcommand\PSH{\mathrm{PSH}}
\newcommand\FS{\mathrm{FS}}
\newcommand\MA{\mathrm{MA}}
\newcommand\Sp{\mathrm{Sp}}
\newcommand\Gr{\mathrm{Gr}}
\newcommand\gr{\mathrm{gr}}
\newcommand\bigN{\mathrm{N}}
\newcommand\diag{\mathrm{diag}}
\newcommand\latt{\mathrm{latt}}

\author{REBOULET Rémi}
\title{THE ASYMPTOTIC FUBINI-STUDY OPERATOR OVER GENERAL NON-ARCHIMEDEAN FIELDS}
\begin{document}

\maketitle

\begin{abstract}\noindent Given an ample line bundle $L$ over a projective $\field$-variety $X$, with $\field$ a non-Archimedean field, we study limits of non-Archimedean metrics on $L$ associated to submultiplicative sequences of norms on the graded pieces of the section ring $R(X,L)$. We show that in a rather general case, the corresponding asymptotic Fubini-Study operator yields a one-to-one correspondence between equivalence classes of bounded graded norms and bounded plurisubharmonic metrics that are regularizable from below. This generalizes results of Boucksom-Jonsson where this problem has been studied in the trivially valued case.
\end{abstract}

\tableofcontents

\newpage

\section*{Introduction.}

Let $X$ be a projective variety defined over a non-Archimedean field $\field$, endowed with a line bundle $L$ which we will assume to be ample. To those objects, one can associate their \textbf{Berkovich analytification}, allowing the use of techniques analogue to those of complex analytic geometry. By considering the non-Archimedean avatar of the space of Kähler potentials: the space $\mathcal{H}(L)$ of \textbf{Fubini-Study potentials}, one can define the class of \textbf{plurisubharmonic} metrics. In this article, we look at the quantization of such plurisubharmonic metrics by sequences of norms acting on sections of tensor powers of $L$, in line with the ideas of \cite{darvasqppt}. In particular, we generalize results in this direction obtained in \cite{ber}, building on the theory developed in \cite{boueri}.

\bsni A central object of study in complex pluripotential theory is the \textbf{space of Kähler potentials}: given a polarized compact connected Kähler manifold $(X/\cc,L)$, and the curvature $\omega_L$ of a reference metric $\phi_0$ on $L$, 
$$\mathcal{H}=\{\phi\in C^\infty(X),\, \omega_L + dd^c\phi>0\}.$$
In non-Archimedean pluripotential theory, where we look at the Berkovich analytifications of $\field$-varieties, we similarly choose a reference metric $\phi_0$ on $L$, which gives a function $s\mapsto |s|_{\phi_0}=|s|$ acting on sections of $L$ (which extends naturally to sections of powers of $L$), then consider the set of \textbf{Fubini-Study potentials} $\mathcal{H}(L)$ whose objects are continuous functions (or \textit{potentials}) $\vp$ from the Berkovich analytification $X^{\mathrm{an}}$ to the reals, of the form
$$\vp=m\mi\max_i\{\log |s_i| + c_i\},$$
where the $(s_i)$ are sections of $mL$ forming a basis of the space $H^0(mL)$, and the $c_i$ are real constants. By allowing decreasing limits, we obtain a much larger class, $\PSH(L)$, that of \textbf{$\phi_0$-plurisubharmonic} functions (or $\phi_0$-\textit{psh} functions).

\bigskip

\noindent Assume from now on $\field$ to be a non-Archimedean field. A simple way to generate Fubini-Study potentials is to consider a norm $\norm$ on $H^0(mL)$. If the absolute value on $\field$ admits a discrete value group, then all such norms are \textbf{diagonalizable}, in the sense that there exists a basis $(s_i)$ of $H^0(mL)$ with
$$\norm\left(\sum_{i=1}^{h^0(mL)}a_i\cdot s_i\right)=\max_i |a_i|\cdot \norm(s_i),$$
where the $a_i$ are elements in $\field$. In the general case, this would only be an inequality. (If $\field$ is densely valued, there is still a way to make sense of this operator, but we will leave this definition to the second section.) The ($m$-th) \textbf{Fubini-Study operator} then sends the norm $\norm$ to the potential
$$\FS_m(\norm)=m\mi\,\max_i\{\log |s_i| - \log \norm(s_i)\}.$$

\bsni More generally, we may consider the class of \textbf{graded norms} on the algebra of multisections $R(X,L)=\bigoplus_{m\geq 1} H^0(mL)$, that is: of sequences $m\mapsto \norm_m$, where for all $m$, $\norm_m$ is a norm on $H^0(mL)$, satisfying a \textit{submultiplicativity condition}: given sections $s_m\in H^0(mL)$, $s_n\in H^0(nL)$,
$$\norm_{m+n}(s_m\cdot s_n)\leq \norm_m(s_m)\cdot \norm_n(s_n).$$
The pointwise limit $\FS_m(\norm_m)$ always exists by Fekete's lemma, but can be infinite in general. To get around this, one restricts to graded norms satisfying the following natural boundedness condition. Note that, by ampleness of $L$, there always exists a positive integer $r$ such that the subalgebra $R(X,rL)$ is generated in degree one. We say that a graded norm $\norm_\bullet$ is \textbf{finitely generated} if there exists such an integer $r$ such that for all positive integers $m$, $\norm_{rm}$ is the quotient norm of induced by $\norm_m$ along the (then surjective) morphism from the $m$-th symmetric powers of $H^0(rL)$ to $H^0(rmL)$. (We define, for a further discussion, a norm \textbf{generated in degree one} to be a finitely generated one with $r=1$, which exists if the algebra $R(X,L)$ is generated in degree one.) We then say that a graded norm is \textbf{bounded} if its distortion with respect to a finitely generated norm is at most exponential as a function of $m$. 

\bsni This growth condition allows us to define an important object: the \textbf{asymptotic spectral measure} of two bounded graded norms $\norm_\bullet$ and $\norm'_\bullet$. Again, we describe it in the case where $\field$ is discretely valued, and refer the reader to the second Section for details on the general case. Since any two diagonalizable norms are diagonalizable in the same basis, we may pick for each $m$ a basis $(s_{m,i})_i$ of $H^0(mL)$ jointly diagonalizing $\norm_m$ and $\norm'_m$, and define the relative spectral measure of $\norm_m$, $\norm'_m$ as the probability measure
$$\sigma_m(\norm_m,\norm'_m)=h^0(mL)\mi \sum_i \delta_{\lambda_{m,i}/m},$$
where $h^0(mL)=\dim H^0(mL)$, and the $\lambda_{m,i}$ are the $h^0(mL)$ elements, counted with multiplicities, of the \textbf{relative spectrum} of $\norm_m$ and $\norm'_m$:
$$\lambda_{m,i}=\log\frac{\norm'_m(s_{m,i})}{\norm_m(s_{m,i})}.$$
One then has from \cite{cmac} that the sequence $\sigma_m(\norm_m,\norm'_m)$ weakly converges to a boundedly supported probability measure, the aforementioned asymptotic spectral measure $\sigma(\norm_\bullet,\norm'_\bullet)$.

\bsni The first absolute moment of this measure defines a semidistance $d_1$ on the space $\mathcal{N}_\bullet(L)$ given for any two bounded graded norms $\norm_\bullet$, $\norm'_\bullet$ by
$$d_1(\norm_\bullet,\norm'_\bullet)=\lim_m (m\cdot h^0(mL))\mi \sum_i |\lambda_{m,i}|.$$
Identifying two norms at zero distance from each other yields an equivalence relation $\sim$ on this space, which may equivalently be characterized as follows: $\norm_\bullet\sim \norm'_\bullet\Leftrightarrow \sigma_\bullet(\norm_\bullet,\norm'_\bullet)=\delta_0$. 

\bsni Going back to pluripotential theory, we may now define the \textbf{asymptotic Fubini-Study operator} on $\mathcal{N}_\bullet(L)$ as the upper semi-continuous regularization
$$\norm_\bullet\mapsto \mathrm{usc}\, \lim_m \FS_m(\norm_m).$$
The existence of the pointwise is ensured by Fekete's lemma, thanks to the submultiplicativity condition and the boundedness of $\norm_\bullet$, and the usc regularization turns out to be plurisubharmonic provided the pair $(X,L)$ satisfies some mild condition, \textbf{continuity of envelopes}, conjectured to always hold as soon as $X$ is normal, and known to be true e.g. for line bundles on smooth varieties defined over a discretely or trivially valued field $\field$ of equal characteristic zero (the important case of the field of Laurent series $\cc((t))$ is such an example; a review of the currently known cases is presented in Example \ref{cpe}). 

\bigskip\noindent This operator maps bounded graded norms to the set of plurisubharmonic functions \textit{regularizable from below} $\mathrm{PSH^\uparrow}$, i.e. psh functions which are limits of increasing nets of Fubini-Study potentials. It is not injective. However, in \cite[Theorem C]{ber}, S. Boucksom and M. Jonsson prove that, if $\field$ is \textit{trivially valued} and of characteristic $0$, this operator descends to an injection from the space of \textit{bounded} graded norms modulo the equivalence relation $\sim$, onto $\mathrm{PSH^\uparrow}$. The main result of this article is a generalization of this statement.

\begin{theorem*}[A]
Assume $(X,L)$ to admit continuity of envelopes. The asymptotic Fubini-Study operator $\FS$ then defines a bijection:
$$\FS: \mathcal{N}_\bullet(L)/\sim\,\to \mathrm{PSH^\uparrow}(L).$$
\end{theorem*}

\noindent The main ingredient in the proof of Theorem A is the following Theorem, which builds on the main result of \cite{ber} and generalizes it:

\begin{theorem*}[B]With the hypotheses of Theorem A, and given two bounded graded norms $\norm_\bullet$, $\norm'_\bullet$, we have that
$$E(\FS(\norm_\bullet),\FS(\norm'_\bullet))=\vol(\norm_\bullet,\norm'_\bullet).$$
\end{theorem*}
\noindent A similar result has been proved, again in the characteristic $0$, trivially valued case, in \cite{ber}.

\bigskip

\noindent The term on the left-hand side is the \textbf{relative Monge-Ampère energy} between two continuous psh functions, which exists thanks to the theory of non-Archimedean Monge-Ampère operators of A. Chambert-Loir and A. Ducros (\cite{cld}), and has been studied for example in \cite{boueri}. 

\bsni The volume, on the right-hand side, is the first moment of the asymptotic spectral measure $\sigma(\norm_\bullet,\norm'_\bullet)$.

\bsni It is a generalization of the relative volumes of balls studied in \cite{bberballs}, and Theorem B is closely related to the results of the latter article. In the aforementioned article, where the base field is $\cc$, one considers two specific norms on spaces of $m$-sections of $L$: having fixed a Hermitian metric $e^{-\phi}$ on $L$, a non-pluripolar compact set $K\subseteq X$, and a probability measure $\mu$ supported on $K$, those are the $L^2$-norm
$$\norm_{2,m\phi}(s)=\left(\int_X |s|^2_{m\phi}\,d\mu\right)^{\frac12}$$
and the $L^\infty$-norm
$$\norm_{\infty,m\phi}(s)=\sup_K |s|_{m\phi}.$$
H. Chen and C. Maclean have studied volumes more generally in \cite{cmac}, and indeed the proof of Theorem B uses their techniques. They rely on \textbf{Okounkov bodies}, which are convex bodies associated to a semigroup, reflecting its asymptotic properties. Originally introduced in \cite{oko}, they have first been studied extensively in \cite{kko} and \cite{lazmus}, and the reader may find a summary of their most important properties in \cite{boubbk}. 

\bigskip

The most important part of the proof of Theorem B, and also of interest as a stand-alone result, is the following Theorem stating that we can recover the volume of two bounded graded norms by approximating it with volumes associated to simpler graded norms:

\begin{theorem*}[C]Let $L$ be such that $R(X,L)$ is generated in degree one. Let $\norm_\bullet$, $\norm'_\bullet$ be two bounded graded norms on $L$, and for each $k\in\nn^*$, let $\norm_\bullet^{(k)}$ and $\norm'_\bullet{}^{(k)}$ denote the graded norms generated in degree one by $\norm_k$ and $\norm'_k$ respectively. Then, we have that:
$$\vol(\norm_\bullet^{(k)},\norm'_\bullet{}^{(k)})\to_{k\to\infty}\vol(\norm_\bullet,\norm'_\bullet).$$
\end{theorem*}

\noindent The reason why this result is useful is that graded norms generated in degree one are simpler to study: their asymptotic behaviour is governed by that of the algebra of sections $R(X,L)$, and by the norm $\norm_1$ on $H^0(L)$.

\bsni The proof of Theorem B also relies on the theory of models of varieties over the valuation ring: to any globally generated \textbf{model} $(\mathcal{X},\mathcal{L})$ of $(X,L)$ as described in e.g. \cite{boueri} is associated a \textbf{lattice norm} on $H^0(L)$, that is, a norm for which there exists an ortho\textit{normal} basis in the non-Archimedean sense. Using a specific construction of such a model $\mathcal{L}$, one obtains a graded norm for which Theorem B essentially holds, thanks to the results of \cite{boueri}. We then pass from this specific case, to the more general case of norms generated in degree one, by approximation of volumes: namely, the set of lattice norms is dense in the set of norms on a fixed vector space, for the $\sup$ (or $d_\infty$) distance. One can then create an approximation by norms "almost" generated in degree one (see Remark \ref{eventuallygdoneremark}), and use Lipschitz continuity of volumes with respect to this distance. Finally, an application of Theorem C concludes the proof.

\paragraph{Plan.} In the \textbf{first section}, we briefly review non-Archimedean fields and Berkovich spaces, then study in depth spaces of norms on vector spaces over \textit{non-trivially valued} non-Archimedean fields. We rely on results from \cite{ber} and \cite{boueri}.

\bigskip

The \textbf{second section} is dedicated to pluripotential theory in the non-Archimedean world, and models, following \cite{boueri}.

\bigskip

The asymptotic properties of spaces of norms are developed in the \textbf{third section}.

\bigskip

In the \textbf{fourth section} are described convex asymptotics, in particular the theory of Okounkov bodies. We show existence of the limit spectral measure, then prove Theorem C. Some references include \cite{boueri}, \cite{boubbk}, \cite{cmac}, while drawing parallels with some constructions of \cite{dwn}. The reader is also invited to consult \cite{boucksomchen}.

\bigskip

In the \textbf{fifth section}, we first prove Theorem B as announced, then follow the steps of \cite{ber} to prove Theorem A.

\section*{Notation.}

$\field$ and $\mathbb{L}$ always denote fields, where our fields are assumed to be commutative.

\bsni A \textbf{variety} $X$ over $\field$ is a geometrically integral, separated scheme, of finite type over $\field$.

\bigskip\noindent Given $L$ a line bundle over a $\field$-variety $X$, $H^0(X,L)$ is identified with the space of sections $\Gamma(X,L)$ of $L$ over $X$. When no confusion can arise, we will often write $H^0(L)$. $h^0(X,L)$ is the $\field$-dimension of the $\field$-vector space $H^0(X,L)$. 

\bigskip

\noindent Given a $\field$-vector space $V$, and a positive integer $n$, $V^{\otimes n}$ refers to the $n$-th tensor power of $V$; $V^{\odot n}$ to the $n$-th symmetric power of $V$; and $V^{\wedge n}$ to the $n$-th exterior power of $V$.

\bigskip

\noindent $\norm$ always refers to a norm on a vector space.

\bsni Unless stated otherwise, $d$ refers to the dimension of some vector space, or of some variety $X$, whichever is natural in context.

\newpage

\section{Norms.}

\subsection{Non-Archimedean fields.}

\begin{defi}\label{defabsvalue}Let $\field$ be a field. A real-valued \textbf{non-Archimedean absolute value} on $\field$ is a function
$$|\cdot|:\field\to \rr_+$$
satisfying the following properties:
\begin{itemize}
\item[(i)] $|0|=0$, $|1|=1$;
\item[(ii)] $|xy|=|x|\cdot|y|$, for $x$, $y\in \field$;
\item[(iii)] $|x+y|\leq\max\{|x|,|y|\}$, for $x$, $y\in \field$.
\end{itemize}
We say that $\field$ is:
\begin{itemize}
\item \textbf{trivially valued} if $|x|=1$ for all $x\neq 0$;
\item \textbf{discretely valued} if $\field$ is not trivially valued, and the set of values of $|\cdot|$ is a discrete subgroup of $\rr_+$;
\item \textbf{densely valued} otherwise (we recall that a subgroup of $\rr_+$ is either discrete or dense).
\end{itemize}
\end{defi}
Note that a real-valued non-Archimedean absolute value enriches $\field$ with a topology.
\begin{defi}Let $\field$ be a field with a real-valued non-Archimedean absolute value. We say that $\field$ is a \textbf{non-Archimedean field} if it is Cauchy complete with respect to the topology induced by its absolute value.
\end{defi}

The reader may find a review of non-Archimedean valued fields and their finer completeness properties in \cite{comisha}.

To any non-Archimedean field $\field$, one may associate the following objects:
\begin{itemize}
\item its \textbf{valuation ring} $\field^\circ$, defined as the set of elements in $\field$ with absolute value $\leq 1$;
\item the \textbf{maximal ideal} of $\field^\circ$, $\field^{\circ\circ}$, characterized as the set of elements in $\field^\circ$ (or $\field$) with absolute value $<1$;
\item its \textbf{residue field} $\tilde{\field}=\field^\circ/\field^{\circ\circ}$.
\end{itemize}

\begin{defi}
Let $\field$ be a non-Archimedean field. If $$\charac{\field}=\charac{\tilde{\field}}=p$$ for some $p$, we say that $\field$ is of \textbf{equicharacteristic} $p$; otherwise, i.e. $\charac{\field}=0$ and $\charac{\tilde\field}=p\geq 2$, we say that $\field$ has \textbf{mixed characteristic}.
\end{defi}

\begin{remark}\label{hahnfield} By \cite{comisha}, a non-trivially valued non-Archimedean field may always be embedded into a field which satisfies the following properties:
\begin{itemize}
\item densely valued;
\item Cauchy complete;
\item algebraically closed.
\end{itemize}
Indeed, let $\hat{\field}$ be the metric completion of the algebraic closure of $\field$. Then $\hat{\field}$ is in fact such a field.
\end{remark}

\subsection{Berkovich analytification.}

Fix a non-Archimedean field $(\field,|\cdot|)$.

\bsni In the complex case, briefly, there exists an \textit{analytification functor} sending a complex projective variety $(X,\oo_X)$ to its associated complex analytic space $(X^{an},\oo_{X^{an}})$, and the categories of coherent sheaves over $X$ and $X^{an}$ are equivalent. This principle fails a priori in the non-Archimedean case when one applies the same constructions \textit{mutatis mutandis}. The category of Berkovich spaces, which we will not describe here but in its most simple cases, serves as a "correct" category for a non-Archimedean GAGA principle. The essential result is the following theorem:

\begin{theorem}Let $(X,\oo_X)$ be a projective variety over $\field$. The analytification functor yields the \textbf{Berkovich analytification} ($X\an,\oo_{X\an})$, which has the following topological properties:
\begin{itemize}
\item connectedness;
\item compactness;
\item Hausdorff.
\end{itemize}
\end{theorem}

The reader is invited to consult the book of V. Berkovich \cite{berko}, or for example \cite{poi}, for an exposition of the theory.

\begin{defi}Let $R$ be any ring. A \textbf{multiplicative seminorm} on $R$ is a function which satisfies the properties (i), (ii), and (iii) of Definition \ref{defabsvalue}.
\end{defi}

\begin{example}[Affine $X\an$]Assume $X=\mathrm{Spec}\, \mathcal{A}$, where $\mathcal{A}$ is an algebra of finite type over $\field$.
\begin{itemize}
\item the set of points of $X\an$ is the set of multiplicative seminorms $|\cdot|_\mathcal{A}$ on $\mathcal{A}$, such that
$$|a|_\mathcal{A}=|a|$$
for all $a\in \field$, that is, $|\cdot|_\mathcal{A}$ extends the absolute value on $\field$;
\item the topology on $X\an$ is the topology of pointwise convergence.
\end{itemize}
In general, given any variety over $\field$, one proceeds by gluing local spaces as constructed above, similarly to the construction of abstract schemes by gluing affine schemes.
\end{example}

\begin{remark}The general construction also applies to Archimedean fields, and in the complex case, the Berkovich analytification of $X$ coincides with the complex analytic space $X\an$.
\end{remark}

\subsection{Spaces of norms.}

Throughout this section, unless otherwise specified, we fix a non-Archimedean field $\field$, with absolute value $|\cdot|$; and a finite-dimensional vector space $V$ over $\field$. Let $d=\dim V$. We follow \cite[Part 1]{boueri}.

\begin{defi}A \textbf{norm} on $V$ is a function
$$\norm:V\to \rr_+,$$
satisfying the following properties:
\begin{itemize}
\item $\norm(v)=0$ if and only if $v=0_V$;
\item $\norm(\lambda\cdot v)=|\lambda|\cdot\norm(v)$, for $\lambda\in \field$, $v\in V$;
\item $\norm(v+w)\leq\max\{\norm(v),\norm(w)\}$, for $v$, $w\in V$.
\end{itemize}
We denote by
$$\mathcal{N}(V)$$
the set of norms on $V$.
\end{defi}

\noindent Norms on a non-Archimedean vector space are stable under the (pointwise) maximum operation, which we denote
$$\norm\vee\norm'=\max(\norm,\norm'),$$
for any two norms $\norm$, $\norm'\in\mathcal{N}(V)$.

\begin{defi}A norm $\norm\in\nv$ is \textbf{diagonalizable} if there exists a basis $(e_1,\dots,e_d)$ of $V$ such that, for all
$$v=\sum v_i e_i,$$
with $v_i\in\field$ for all $i$, we have that
$$\norm(v)=\max_i |v_i|\cdot \norm(e_i).$$
We say that it is a \textbf{lattice norm}, or a \textbf{pure diagonalizable norm}, if, for all $i$, $\norm(e_i)=1$.
We define
$$\mathcal{N}^{\diag}(V),\,\mathcal{N}^{\latt}(V)$$
as respectively the set of diagonalizable norms and the set of lattice norms on $V$.
\end{defi}

\noindent It is also common practice to use the terminology of \textbf{cartesian bases}, see e.g. \cite[Ch. 2]{bgr}.

\begin{remark}
We define a \textbf{lattice} of $V$ as a finite submodule $L$ of $V$ over $\field^o$, such that $L\otimes_{\field^o}\field=V$. In particular, the unit ball of a lattice norm is always a lattice. 
\end{remark}

\begin{remark}One may define norms as follows.
\begin{itemize}
\item let $W\subset V$ be a subspace; any norm $\norm$ induces a quotient norm $\norm_{V/W}$ on $V/W$;
\item a norm $\norm$ on $V$ induces a norm $\norm^{\otimes n}$ on any tensor power $V^{\otimes n}$ of $V$.
\end{itemize}
We may combine these constructions. Let $n$ be an integer, $\lambda$ be a partition of $n$, and let $\mathbb{S}^\lambda$ denote the Schur functor associated to $\lambda$. Then a norm $\norm\in\nv$ defines a norm $\norm^\lambda\in\mathcal{N}(\mathbb{S}^\lambda(V))$, as this vector space is a composition of quotients of tensor products. In particular, $\norm\in\nv$ induces:
\begin{itemize}
\item a norm $\norm^{\wedge n}$ on the $n$-fold exterior product $V^{\wedge n}$;
\item a norm $\norm^{\odot n}$ on the $n$-fold symmetric product $V^{\odot n}$.
\end{itemize}
\end{remark}

More can be said about the structure of $\nv$. To any basis $\{e_1,\dots,e_d\}$ of $V$, one may associate three objects. Fix one such basis.

\begin{defi}The \textbf{apartment} associated to the basis $(e_i)$ is the following set of diagonalizable norms:
$$\mathcal{N}^{\diag}(V)\supset\mathbb{A}_{(e_i)}=\{\norm\in\mathcal{N}^{\diag}(v),\,\norm\text{ is diagonalizable in the basis }(e_i)\}.$$
We then define the \textbf{injection map} associated to $(e_i)$:
$$i_{(e_i)}:\rr^d\to \mathcal{N}^{\diag}(V),$$
sending a real vector $(r_i)$ to the unique norm $\norm$ in $\mathbb{A}_{(e_i)}$ satisfying
$$\norm(e_i)=e^{-r_i}$$
for all $i$.
\end{defi}

\subsection{Spectra.}

We set $\norm$, $\norm'\in\nv$. We introduce the notion of (relative) spectrum of two norms, and its associated spectral measure. These constructions will enable us to define easily many operations on tuples of norms, such as distances, which will be useful in further study of spaces of norms.

\begin{defi}The \textbf{relative spectrum} of $\norm$ and $\norm'$ is the set (with multiplicities) $\Sp(\norm,\norm')$ which contains all reals of the form
$$\lambda_i(\norm,\norm')=\sup_{W\in \bigcup_{i\leq r \leq \dim V} \Gr_\field(r,V)} \inf_{w\in W-\{0\}}[\log \norm'(w) - \log \norm(w)],$$
where
$$\Gr_\field(r,V)$$
denotes the $r$-th Grassmannian of $V$.
\end{defi}

\begin{remark}By \cite[P2.24]{boueri}, if both norms are diagonalizable by a basis $(s_i)$, ordered such that $$i>j \Rightarrow \frac{\norm'(s_i)}{\norm(s_i)} \geq \frac{\norm'(s_j)}{\norm(s_j)},$$
then
$$\lambda_i(\norm,\norm')=\log \norm'(s_i)-\log \norm(s_i).$$
\end{remark}

\begin{defi}The \textbf{relative spectral measure}
$$\sigma(\norm,\norm')$$
of $\norm$ and $\norm'$ is defined to be discrete probability measure supported on $\Sp(\norm,\norm')$, that is:
$$\sigma(\norm,\norm')=d\mi \sum \delta_{\lambda_i(\norm,\norm')},$$
where we recall that $d=\dim_\field V$.
\end{defi}

\begin{defi}Let $p\in[1,\infty)$. The $d_p$-distance of $\norm$ and $\norm'$ is defined by:
$$d_p(\norm,\norm')^p=\int_\rr |\lambda|^p\,d\sigma(\norm,\norm').$$
We furthermore define
$$d_\infty(\norm,\norm')=\max \Sp(\norm,\norm').$$
In this paper, we will only use the distances $d_1$ and $d_\infty$, which have more practical expressions:
$$d_\infty(\norm,\norm')=\sup_{v\in V-\{0\}}|\log\norm'(v)-\log\norm(v)|,$$
and
$$d_1(\norm,\norm')=d\mi\sum \lambda_i(\norm,\norm'),$$
where $d=\dim_\field V$.
\end{defi}

\begin{remark}[Important characterization of the distance $d_\infty$]\label{dinftydistorsion}
The distance $d_\infty(\norm,\norm')$ is equivalently characterized as the \textbf{best constant} $C>0$ such that for all $v\in V$,
$$e^{-C}\norm'(v)\leq \norm(v) \leq e^C \norm(v).$$
\end{remark}

\bigskip

Closely related to the distance $d_1$ (see Theorem \ref{volumeslog} below), the following definition generalizes ratios of volumes of balls of holomorphic sections, originally studied in \cite{bberballs}.

\begin{defi}The \textbf{relative volume} of $\norm$ and $\norm'$ is defined by:
$$\vol(\norm,\norm')=\int_\rr \lambda\,d\sigma(\norm,\norm'),$$
that is: the mean value of the relative spectrum of those norms. Note that we normalize by $d$ (by our definition of the relative measure), while the authors in \cite[T2.25]{boueri} do not.
\end{defi}

\begin{remark}In particular, if $\norm\leq\norm'$, then
$$\vol(\norm,\norm')=d_1(\norm,\norm'),$$
and, reversing the inequality, we obtain
$$-\vol(\norm,\norm')=d_1(\norm,\norm').$$
\end{remark}

\begin{theorem}[{\cite[T2.25]{boueri}}]\label{volumeslog}Let $\{e_1,\dots,e_d\}$ be a basis of $V$. We have that
$$\vol(\norm,\norm')=\frac{1}{d}\left[\log{\norm'}^{\wedge d}(e_1\wedge\dots\wedge e_n) - \log\norm^{\wedge d}(e_1\wedge\dots\wedge e_n)\right]$$
\end{theorem}

\begin{corollary}Volumes satisfy a \textbf{cocycle property}: given a third norm $\norm''\in\nv$,
$$\vol(\norm,\norm')=\vol(\norm,\norm'')+\vol(\norm'',\norm').$$
\end{corollary}

\begin{prop}[{\cite[P1.8, T1.19, L1.29]{boueri}}]Consider $\nv$ (and its subspaces) endowed with the distance $d_\infty$. We then have that:
\begin{itemize}
\item[(i)] $\nv$ is complete;
\item[(ii)] $\mathcal{N}^{\diag}(V)$ is dense in $\mathcal{N}(V)$, with equality if $\field$ is discretely valued;
\item[(iii)] if $\field$ is discretely valued, $\mathcal{N}^{\latt}(V)$ is discrete and closed in $\mathcal{N}(V)$;
\item[(iv)] if $\field$ is densely valued, $\mathcal{N}^{\latt}(V)$ is dense in $\mathcal{N}^{\diag}(V)$.
\end{itemize}
\end{prop}

We then have that:

\begin{lemma}[{\cite[L1.29]{boueri}}]Let $\field$ be nontrivially valued.
\begin{itemize}
\item if $\field$ is discretely valued, the unit ball of any diagonalizable norm is a lattice of $V$;
\item if $\field$ is densely valued, the unit ball of a norm $\norm\in\nv$ is a lattice if and only if $\norm$ is a lattice norm.
\end{itemize}
\end{lemma}

\begin{remark}As seen above, the discrete and densely valued cases are fundamentally different. We will often use Remark \ref{hahnfield} and embed $\field$ into an algebraically closed field. We now study the behaviour of the objects previously defined under a general field extension.
\end{remark}

\begin{defi}
Let $\mathbb{L}/\field$ be a non-Archimedean field extension. Let $\norm$ be a non-Archimedean norm on $V=V_\field$. The \textbf{ground field extension} $\norm_{\mathbb{L}}$ on $V_\mathbb{L}=V\otimes_\field \mathbb{L}$ is defined as
$$\norm_{\mathbb{L}}(v')=\inf \max_i |a'_i|\cdot \norm(v_i),$$
for any $v'\in V_\mathbb{L}$, where the $\inf$ is defined over all representations
$$v'=\sum_i a'_i \cdot v_i,$$
with coefficients $a'_i$ in $\mathbb{L}$ and $v_i\in V_\field$.
\end{defi}

\noindent This defines by \cite[P1.24(i)]{boueri} a non-Archimedean norm on $V_\mathbb{L}$, which coincides with the original norm $\norm$ on $V_\field$. The two essential results for us are the following:

\begin{prop}[{\cite[L1.25,P2.14(v)]{boueri}}]\label{extvol}Let $\mathbb{L}/\field$ be a field extension. Let $\norm$ be a norm on $V=V_\field$, with ground field extension $\norm_{\mathbb{L}}$ on $V_\mathbb{L}$. We then have:
\begin{itemize}
\item if $\norm$ is diagonalizable with basis $(e_i)$, then $\norm_{\mathbb{L}}$ is also diagonalizable with basis $(e_i\otimes 1)$;
\item for any other norm $\norm'$ with ground field extension $\norm'_{\mathbb{L}}$, we have
$$\vol(\norm,\norm')=\vol(\norm_{\mathbb{L}},\norm'_{\mathbb{L}}).$$
\end{itemize}
\end{prop}
\noindent Finally, we note that relative volumes behave well with respect to the $d_\infty$ distance.
\begin{lemma}[Volumes are Lipschitz, {\cite[P2.14]{boueri}}]\label{volumelipschitz}Fix a norm $\norm\in\nv$. We then have that:
$$|\vol(\norm',\norm)-\vol(\norm'',\norm)|\leq d_\infty(\norm',\norm''),$$
and similarly in the second variable. (Note that, due to our normalization, the Lipschitz coefficient is $1$.)
\end{lemma}

\section{Line bundles.}

\subsection{Metrics.}

In this section, we consider a projective $\field$-variety $X$, and $p:L\to X$ a line bundle.

\begin{defi}Given a point $x\in X\an$, we denote by $\mathcal{H}(x)$ the completion of the residue field at $x$, with its canonical absolute value. A \textbf{continuous metric} $\phi$ on $L$ is a family of functions $$\phi_x:L\otimes \mathcal{H}(x)\to \rr\cup\{\infty\},$$
for all $x\in X$, such that $|\cdot|_{\phi_x}=e^{-\phi_x}$ is a norm on the $\mathcal{H}(x)$-line $L\otimes \mathcal{H}(x)$, and such that for any local section $s_U\in H^0(U,L)$ over a Zariski open set $U$, the composition
$$|s_U|_\phi:U\xrightarrow{s_U}L|_U\xrightarrow{|\cdot|_\phi}\rr_+$$
is continuous. We refer the reader to \cite{boueri} for further details.
We say that a continuous metric on $L$ lies in the set $\mathcal{C}^0(L)$.
\end{defi}

\noindent We use the conventions from \cite{boueri}, which we recall here:
\begin{itemize}
\item the tensor product of line bundles is denoted additively: $L^{\otimes k}\otimes M^{-1}$ is written as $kL-M$;
\item given metrics $\phi$, resp. $\phi'$ on $L$, resp. $M$, the induced metric on $kL-M$ is written as $k\phi-\phi'$;
\item we identify $\mathcal{C}^0(\oo_X)$ with $C^0(X)$ via $\varphi \leftrightarrow -\log |1|_\phi$, whence $\mathcal{C}^0(L)$ is an affine space modelled on $C^0(X)$, for $\phi,\,\phi'\in \mathcal{C}^0(L)$ transform as $|\cdot|_\phi=|\cdot |_\psi e^{\phi'-\phi}$.
\end{itemize}

\begin{remark}In general, we may define a non-necessarily continuous metric over $L$ by removing the second condition above, or impose different regularity conditions. In this case, the last point above still holds, modulo the necessary modifications of the sets of functions being used.
\end{remark}

\subsection{From norms to metrics and back: $\FS$ and $\mathrm{N}$.}

We fix a reference metric $\phi_0$ on $L$, and denote
$$|s|=|s|_{\phi_0},$$
given a section $s\in H^0(L)$. This will allow us to define certain operators relating norms and metrics, which always require such a choice of a metric, without explicitly stating that we do so. 

\noindent We first define the non-Archimedean equivalent of Kähler potentials:
\begin{defi}
Let $m$ be a large enough integer such that $mL$ is globally generated. A \textbf{Fubini-Study potential} on $L$ is a continuous function $f:X^{an}\to\rr$ of the form
$$f=\frac{1}{m}\log\left(\sup_{s\in H^0(mL)-\{0\}} \frac{|s|}{\norm(s)}\right),$$
where $\norm\in\mathcal{N}(H^0(mL))$. The set of Fubini-Study potentials is denoted $\mathcal{H}$.
\end{defi}

\noindent In order to simplify notation, we may introduce so-called \textit{Fubini-Study operators}, which give a way to turn norms into metrics:
\begin{defi}
Let $m$ be such that $mL$ is globally generated. We define the ($m$-th) \textbf{Fubini-Study operator} as such:
\begin{align*}
\FS_m:\mathcal{N}(H^0(mL))&\to C^0(X^{an}),\\
\norm&\mapsto \FS_m(\norm)=\frac{1}{m}\log\sup_{s\in H^0(mL)-\{0\}} \frac{|s|}{\norm(s)}.
\end{align*}
\end{defi}

\begin{remark}It is immediate that the set of Fubini-Study potentials on $L$ may be realized as the union of the images of the operators $\FS_m$ for all $m$ such that these operators are defined.
\end{remark}

\begin{remark}There is a nice way to compute the Fubini-Study operators restricted to each $\mathcal{N}^{\diag}(H^0(mL))$. Indeed, assume $(s_i)$ diagonalizes a norm $\norm$ on $H^0(mL)$ for some large $m$. We then have that:
$$\FS_m(\norm)=\frac{1}{m}\log \max_i \frac{|s_i|}{\norm(s_i)}.$$
\end{remark}

\begin{defi}
A metric on $L$ is \textbf{Fubini-Study} if its associated potential is of the form
$$\phi=\frac{1}{m}\log \max_i a_i\cdot |s_i|,$$
for some $m> 0$, $a_i\in\rr_{>0}$, and some basis $(s_i)$ of $H^0(mL)$. We say furthermore that a Fubini-Study metric is a \textbf{pure} \textit{Fubini-Study metric} if all of the coefficients $a_i$ above may be chosen to be $1$.
\end{defi}

\begin{remark}It follows that pure Fubini-Study metrics are characterized as those metrics in the image of some Fubini-Study operator restricted to the set of orthonormal (or \textit{lattice}) norms.
\end{remark}

\begin{defi}A reference metric $\phi_0$ having been fixed, a function on $X$ is said to be \textbf{$\phi_0$-plurisubharmonic} (or $\phi_0$-\textbf{psh}) if it is a limit of a decreasing net (equivalently, if it is continuous, a uniform limit) of functions in $\mathcal{H}$. In particular, a $\phi_0$-psh function is usc. A metric on $L$ is \textbf{psh} if its associated potential is psh. We denote the set of $\phi_0$-psh functions on $X$ by $\PSH_{\phi_0}(X)$ and the set of PSH metrics on $L$ by $\PSH(L)$. When it is clear from context, we will only write $\PSH$, and speak of plurisubharmonic functions, rather than $\phi_0$-plurisubharmonic functions.
\end{defi}

\begin{remark}The set $\PSH_{\phi_0}(X)\cup \{-\infty\}$ is stable under the following operations:
\begin{itemize}
\item addition of a real constant;
\item limits of decreasing nets;
\item maximum of finitely many functions.
\end{itemize}
\end{remark}

We now construct an operator from the space of continuous functions $C^0(X)$ (on $X^{an}$) to the space $\mathcal{N}_\bullet(L)$.

\begin{defi}The \textbf{$m$-th supnorm operator} $\bigN_m$ sends $\vp\in C^0(X)$ to the norm on $H^0(mL)$ defined as
$$\bigN_m(\vp)(s_m)=\sup_{X}|s_m|_{m\vp},$$
for $s_m\in H^0(mL)$.
\end{defi}

\noindent It it straightforward from the definitions to see that:
\begin{prop}[Lipschitz-like properties of $\FS$ and $N$]Consider two bounded functions $\vp$ and $\vp'\in C^0(X)$, and two norms $\norm$ and $\norm'$ on some $H^0(mL)$, $m\geq 1$. We then have that:
\begin{itemize}
\item $d_\infty(\bigN_m(\vp),\bigN_m(\vp'))\leq m\cdot\sup_X |\vp'-\vp|$;
\item $\sup_X |\FS(\norm') - \FS(\norm)|\leq m\mi \cdot d_\infty(\norm,\norm')$.
\end{itemize}
\end{prop}

For later use, we also define:
\begin{defi}
The \textbf{graded supnorm operator}
$$\bigN_\bullet:C^0(X)\to \mathcal{N}_\bullet(L)$$
sends $\vp$ to the sequence of norms $(\bigN_m(\vp))_m$.
\end{defi}
We will say that a sequence of norms is \textit{the $\sup$ norm associated to a bounded metric $\phi$} if and only if it is the image of the potential associated to $\phi$ under the operator $\bigN_\bullet$.

\subsection{Monge-Ampère measures.}

The theory of Chambert-Loir-Ducros (\cite{cld}) shows that, as in the complex case, we may associate to a $d$-uple of continuous, plurisubharmonic metrics $(\phi_1,\dots,\phi_d)$ on line bundles $L_i\to X$, $i\in\{1,\dots,d\}$ (with $d=\dim X$), a positive Radon measure
$$dd^c\phi_1\wedge\dots\wedge dd^c\phi_d,$$
called the \textbf{mixed Monge-Ampère measure} of $(\phi_1,\dots,\phi_d)$. Its total mass is equal to the intersection number of the $L_i$'s. The association of such a $d$-uple to its mixed Monge-Ampère measure is naturally symmetric and additive in each variable. It furthermore is stable under ground field extension:

\begin{prop}[MMA measures are invariant under ground field extension, {\cite[P8.3(iv)]{boueri}}]\label{extma}Let $\mathbb{L}/\field$ be a non-Archimedean field extension. Consider the cartesian diagram:

\begin{center}\begin{tikzcd}X\an_{\mathbb{L}}=(X\times_{\Spec{\field}}\Spec{\mathbb{L}})\an \arrow[r, "\pi_1"] \arrow[d,"\pi_2"] & X\an \arrow[d] \\
\Spec{\mathbb{L}}\an \arrow[r] & \Spec{\field}\an
\end{tikzcd}\end{center}

\noindent We then have that:
$${\pi_1}_*\left(dd^c(\pi_1^*\phi_1)\wedge\dots\wedge dd^c(\pi_1^*\phi_d)\right)=dd^c\phi_1\wedge\dots\wedge dd^c\phi_d.$$
\end{prop}

\begin{defi}Assume $L$ to be a semiample and big line bundle. Let $\phi$, $\phi'\in \PSH(L)$. The \textbf{Monge-Ampère measure} of $\phi$ is the mixed Monge-Ampère measure
$$\MA(\phi)=\frac{1}{\vol(L)}\cdot (dd^c\phi)^d,$$
where $\vol(L)=c_1(L)^d$. It is a Radon probability measure. We define the \textbf{relative Monge-Ampère energy} of $\phi$ and $\phi'$ as follows:
$$E(\phi,\phi')=\frac{1}{(d+1)\vol(L)}\sum_{i=0}^d\int_X (\phi-\phi')\,(dd^c\phi)^i\wedge (dd^c\phi')^{d-i}.$$
\end{defi}

\begin{remark}\label{remarkenergies}
Our conventions are slightly different than those of \cite{boueri} - we adopt those of \cite{ber}: our Monge-Ampère energy is normalized by the volume of the semiample line bundle $L$.
\end{remark}

\noindent Much like volumes of norms, the Monge-Ampère energy satisfies a cocycle condition:
\begin{prop}[{\cite[P9.14(i)]{boueri}}]Let $\phi$, $\phi'$ and $\phi''\in \PSH(L)$, we have that:
$$E(\phi,\phi')=E(\phi',\phi'') + E(\phi'',\phi).$$
\end{prop}

\subsection{Models.}

From now on, we consider an \textit{ample} line bundle $L\to X$.

\begin{defi}
A \textbf{model} of $X$ is the data of a flat scheme $\mathcal{X}$ of finite type over the valuation ring $\field^\circ$, and an isomorphism $\mathcal{X}\times_{\field^\circ}\field\to X$ as schemes over $\field$.
\end{defi}

\begin{remark}As a scheme, $\Spec{\field^\circ}$ has two points: the \textbf{generic point} corresponding to the ideal $\{0\}$, with residue field $\field$, and a closed point, the \textbf{special point}, corresponding to $\field^{\circ\circ}$, with residue field $\tilde{\field}$. Changing the base to $\field$ (resp. $\tilde{\field}$) amounts to taking the \textbf{generic fiber} $\mathcal{X}_\field$ (resp. \textbf{special fiber} $\mathcal{X}_s$).
\end{remark}

\begin{remark}Note that, if $\field$ is discretely valued, $\field^\circ$ is a discrete valuation ring. On the other hand, if $\field$ is densely valued, $\field^\circ$ can never be Noetherian. This raises obstacles to prove certain results in this case, such as continuity of envelopes, a property which we will encounter later on.
\end{remark}

\begin{example}If $\field$ is trivially valued, the only model of $X$ is $X$ itself. This is one reason why the authors in \cite{ber} consider the notion of \textbf{test configuration}, which serves as an alternative notion of models in this specific case.
\end{example}

\begin{defi}Let $L$ be a line bundle on $X$. A \textbf{model} $(\mathcal{X},\mathcal{L})$ of $(X,L)$ consists in a model $\mathcal{X}$ of $X$, projective over $\field^\circ$, and a line bundle $\mathcal{L}$ on $\mathcal{X}$ extending $L$ with respect to the identification $\mathcal{X}_\field\simeq X$. One then says that $\mathcal{L}$ is a model of $L$ determined on $\mathcal{X}$.
\end{defi}

\begin{defi}Let $\mathcal{X}$ be a model of $X$. There exists a compact subset $X^\beth\subseteq X\an$ (the Hebrew letter beth $\beth$ is to be read as "bet"), constructed in the affine case $\mathcal{X}=\Spec{A}$, with $A$ an algebra of finite type over $\field^\circ$, as the set of analytic points $x\in X^{an}$ with $|a(x)|\leq 1$, for any $a\in A$. The non-affine case is constructed by gluing affine open sets.
\end{defi}
\noindent Further discussion on the properties of the set $X^\beth$ may be found in \cite[4.3]{boueri}. For our purposes, the important result is the following:
\begin{lemma}[{\cite[L4.6]{boueri}}]If the model $\mathcal{X}$ is proper over $\field^o$, we have that
$$X\an=X^\beth.$$
\end{lemma}

\begin{defi}
Let $(\mathcal{X},\mathcal{L})$ be a model of $(X,L)$. Let $(\mathcal{U}_i,\tau_i)$ be a trivialization of $\mathcal{L}$ over a finite open cover of $\mathcal{X}$. By the above lemma, the sets $\mathcal{U}_i^\beth$ cover $X^{an}=X^\beth$, so that we may define the \textbf{model metric} $\phi_\mathcal{L}$ associated to $\mathcal{L}$ by requiring the trivializations $\tau_i$ to be unitary for $\phi_\mathcal{L}$ over each $\mathcal{U}_i^\beth$. This metric is then well-defined and unique.
\end{defi}

\noindent If $L$ is merely a $\qq$-line bundle, the above definitions still make sense:
\begin{defi}
Let $L$ be a $\qq$-line bundle on $X$, and $\mathcal{X}$ a projective model of $X$. A \textbf{$\qq$-model} $\mathcal{L}$ of $L$ determined on $\mathcal{X}$ is then a $\qq$-line bundle $\mathcal{L}$ on $\mathcal{X}$, such that for some positive integer $m$, $m\mathcal{L}$ is a model of $mL$ determined on $\mathcal{X}$. Its associated model metric is then defined as $\phi_\mathcal{L}=m\mi \phi_{m\mathcal{L}}$, and we note by \cite[L5.10]{boueri} that $\phi_\mathcal{L}$ is independent of the choice of such an $m$.
\end{defi}

\noindent We have the following useful result:
\begin{prop}[{\cite[T5.14]{boueri}}]A continuous metric $\phi$ on a line bundle $L$ is a pure Fubini-Study metric if and only if it is a model metric associated to some semiample $\qq$-model $\mathcal{L}$ of $L$.
\end{prop}

\begin{example}\label{computingmodelmetric}An essential point in the proof of the previous result is the fact that the space of sections of a model of $L$ is a \textit{lattice} in $H^0(L)$, hence determines a lattice norm on $H^0(L)$, usually denoted $\norm_{H^0(\mathcal{L})}$. By \cite[E7.20]{boueri}, for all $m$ such that $m\mathcal{L}$ is globally generated, we have:
$$\phi=m\mi \FS(\norm_{H^0(m\mathcal{L})}).$$
\end{example}

\section{Graded norms and the asymptotic spectral measure.}

In this section, we assume $\field$ to be any non-Archimedean field, and $L$ to be a semiample line bundle over a variety $X/\field$. We let $$R(X,L)=\bigoplus_m H^0(X,mL)$$ be the graded algebra of multisections of $L$, which is integral. 

\subsection{Bounded graded norms.}\label{phi}

We define a natural notion of norm on $R(X,L)$:

\begin{defi}
A \textbf{graded norm} on $R(X,L)$ is the data of norms $\norm_\bullet=(\norm_m)$ on each graded piece $H^0(mL)$. Such a graded norm is said to be \textbf{submultiplicative} if:
\begin{itemize}
\item for any $s_m\in H^0(mL)$, $s_n\in H^0(nL)$, we have that
$$\norm_{m+n}(s_m\cdot s_n)\leq \norm_m(s_m)\cdot\norm_n(s_n);$$
\end{itemize}
and \textbf{multiplicative} if the above inequality is always an equality.
\end{defi}

\begin{example}
If $(\mathcal{X},\mathcal{L})$ is a model of $(X,L)$, the sequence of norms $\norm_{H^0(\bullet\mathcal{L})}$ is a submultiplicative graded norm. In particular, if $\field$ is trivially valued, then the sequence of trivial norms $\norm_{\bullet,triv}$ is a submultiplicative graded norm. 
\end{example}

\begin{defi}\label{defmodelnorm}
We say that a graded norm $\norm_\bullet$ is a \textbf{model graded norm} on $L$ if there exists a model $(\mathcal{X},\mathcal{L})$ of $(X,L)$, such that for all $m$, 
$$\norm_m=\norm_{H^0(m\mathcal{L})}.$$
\end{defi}

In order to endow spaces of graded norms with metric structures, we would like to associate, to any two submultiplicative graded norms $\norm_\bullet$, $\norm'_\bullet$ as defined above, a "spectral measure", in analogy with the spectral measure associated to two norms on a fixed vector space. The 'intuitive' way to do this, is to consider a weak limit in some sense of the (rescaled) measures
$$m\mi_*\sigma(\norm_m,\norm'_m).$$
It turns out that there exists such a limit, but we have to enforce a geometric growth condition on our norms. We will see later that it can also be replaced with another simpler growth condition.

\begin{defi}A submultiplicative graded norm $\norm_\bullet$ is said to be \textbf{bounded} if and only if there exists a model $(\mathcal{X},\mathcal{L})$ of $(X,L)$, with associated graded norm
$$\norm_{H^0(\bullet\mathcal{L})},$$
such that
\begin{equation}\label{linearequivalenceclass}
\lim_m m\mi d_\infty(\norm_m,\norm_{H^0(m\mathcal{L})})<\infty.
\end{equation}
We denote by $\mathcal{N}_\bullet(L)$ the set of all bounded submultiplicative graded norms on $R(X,L)$.
\end{defi}

\noindent In other words, a bounded graded norm has at most exponentially linear distorsion in $m$ compared to a model graded norm: there exists a constant $C>0$ and a model $(\mathcal{X},\mathcal{L})$ of $(X,L)$ such that, for all $m$ large enough,
$$e^{-mC}\norm_{H^0(m\mathcal{L})}\leq \norm_m \leq e^{mC}\norm_{H^0(m\mathcal{L})},$$
in light of Remark \ref{dinftydistorsion}.

\begin{remark}
Note that, if the Condition (\ref{linearequivalenceclass}) holds for one model, then it holds for all models, and all graded supnorms associated to bounded metrics on $L$, see \cite[T6.4]{boueri}. In particular, Condition (\ref{linearequivalenceclass}) in the case where $\field$ is trivially valued amounds to requiring our norm to be exponentially linearly close to the graded tricial norm $\chi_{triv,\bullet}$.
\end{remark}

As stated above, we then have the existence of a "limit" spectral measure for elements in $\mathcal{N}_\bullet(L)$. We will prove this result in the next section; it relies on Okounkov bodies. The exposition follows \cite{cmac} and \cite[T9.5]{boueri}. For now, we only state it:
\begin{theorem}\label{existslimitmeasure}
Let $\norm_\bullet$ and $\norm'_\bullet$ be bounded submultiplicative graded norms on $V_\bullet$. 
The sequence of spectral measures
$$m\mi_*\sigma(\norm_m,\norm'_m)$$
has uniformly bounded support and converges weakly to a boundedly supported limit measure
$$\sigma(\norm_\bullet,\norm'_\bullet).$$
\end{theorem}

\begin{defi}The \textbf{asymptotic spectral measure} of $\norm_\bullet$ and $\norm'_\bullet$ is defined as the limit measure in Theorem \ref{existslimitmeasure}:
$$\sigma(\norm_\bullet,\norm'_\bullet).$$
\end{defi}

\subsection{Metric structures on spaces of graded norms.}

We may now introduce an asymptotic version of the distance $d_1$ we have previously defined on a space of norms over a fixed finite-dimensional $\field$-vector space. We similarly construct asymptotic volumes and $d_\infty$ distances.

\begin{defi}Set $\norm_\bullet$, $\norm'_\bullet\in\mathcal{N}_\bullet(V_\bullet)$.
\begin{itemize}
\item their \textbf{asymptotic $d_1$-distance} is defined by:
$$d_1(\norm_\bullet,\norm'_\bullet)=\int_\rr|\lambda|\,d\sigma(\norm_\bullet,\norm'_\bullet);$$
\item their \textbf{asymptotic $d_\infty$-distance} is defined by:
$$d_\infty(\norm_\bullet,\norm'_\bullet)=\sup_{m\in\nn^*}m\mi d_\infty(\norm_m,\norm'_m);$$
\item their \textbf{asymptotic relative volume} is defined by:
$$\vol(\norm_\bullet,\norm'_\bullet)=\int_\rr\lambda\,d\sigma(\norm_\bullet,\norm'_\bullet).$$
\end{itemize} 
\end{defi}

\begin{remark}\label{remarkvolumes}
Again, we adopt here the conventions of \cite{ber} concerning the volume: in \cite{boueri}, the volume (let us denote it $\vol_{BE}$) of two norms on a fixed space is not normalized by the dimension of the vector space, while they define the asymptotic volume as
$$\lim_m \frac{(\dim X)!}{m^{\dim X+1}}\vol_{BE}(\norm_m,\norm'_m),$$
which is then equal to $\vol(L)$ times the asymptotic volume presently normalized.
\end{remark}

\bigskip Note that the $d_1$ "distance" above is merely a semidistance: for example, since for any two norms $\norm$, $\norm'$ on a fixed $\field$-vector space $V$,
$$d_1(\norm,\norm')\leq d_\infty(\norm,\norm'),$$
so that if two bounded graded norms have at most subexponential growth in $m$, 
$$d_1(\norm_\bullet,\norm'_\bullet)\leq \lim_m m\mi d_\infty(\norm_m,\norm'_m)=0.$$
Even worse: there can be bounded graded norms such that $m\mi d_\infty(\norm_m,\norm'_m)\to C>0$ but $d_1(\norm_\bullet,\norm'_\bullet)=0$, see e.g. \cite[R3.8]{ber}. We then have to identify such norms:
\begin{defi}Two bounded graded norms $\norm_\bullet$ and $\norm'_\bullet$ are \textbf{asymptotically equivalent}, and we write
$$\norm_\bullet\sim \norm'_\bullet,$$
if and only if
$$d_1(\norm_\bullet,\norm'_\bullet)=0.$$
We then have that
$$(\mathcal{N}_\bullet(L)/\sim,d_1)$$
with the induced $d_1$ distance, is a \textit{bona fide} metric space.
\end{defi}

\begin{remark}
Note that, since $d_\infty$ is defined as a $\sup$ rather than as a limit, it is a genuine distance on $\mathcal{N}_\bullet(L)$.
\end{remark}

\noindent It is important to remark that, by definition of the limit measure, we have:
$$\vol(\norm_\bullet,\norm'_\bullet)=\lim_m m\mi \vol(\norm_m,\norm'_m)$$
and
$$d_1(\norm_\bullet,\norm'_\bullet)=\lim_m m\mi d_1(\norm_m,\norm'_m)$$
(but this is not the case for $d_\infty$).

\begin{remark}One has that $\norm_\bullet\sim \norm'_\bullet$ if one of the three following equivalent conditions is realized:
\begin{itemize}
\item for some $p\in[1,\infty)$, $\int_\rr|\lambda|^p\,d\sigma(\norm_\bullet,\norm'_\bullet)=0$;
\item for all $p\in[1,\infty)$, $\int_\rr|\lambda|^p\,d\sigma(\norm_\bullet,\norm'_\bullet)=0$;
\item the asymptotic spectral measure $\sigma(\norm_\bullet,\norm'_\bullet)$ is the Dirac measure $\delta_0$.
\end{itemize}
\end{remark}

\bigskip

\noindent Finally, following the idea that we will consider graded norms on line bundles over any non-Archimedean fields, and need to pass to an algebraically closed extension to apply Okounkov body techniques, it is important to check that graded norms stay graded after piecewise ground field extension. This is a result of Boucksom-Eriksson:

\begin{lemma}[{\cite[L9.4]{boueri}}]\label{extvolasymptotic}Set $\norm_\bullet$, $\norm'_\bullet\in\mathcal{N}_\bullet(L)$. Let $\mathbb{L}/\field$ be a complete field extension, and consider the sequences of ground field extensions
$$\norm_{\mathbb{L},\bullet},\,\norm'_{\mathbb{L},\bullet}.$$
Then, those sequences are bounded graded norms, and furthermore
$$\vol(\norm_{\mathbb{L},\bullet},\norm'_{\mathbb{L},\bullet})=\vol(\norm_\bullet,\norm'_\bullet).$$
\end{lemma}

\subsection{Norms generated in degree one.}

\noindent We introduce a subspace of $\mathcal{N}_\bullet(L)$ in the case where the algebra of sections of $L$ satisfies good properties with respect to symmetry morphisms.

\begin{defi}
Let $L$ be a semiample line bundle over $X$. We say that $R(X,L)$ is \textbf{generated in degree one} if for all $r\in\mathbb{N}$, the morphism induced by $r$-symmetric powers: 
$$\Phi_r: H^0(L)^{\odot r} \to H^0(rL)$$
is surjective.
\end{defi}

\noindent This definition then propagates to bounded graded norms:

\begin{defi}
Assume $R(X,L)$ to be generated in degree one. We say that $\norm_\bullet\in\mathcal{N}_\bullet(L)$ is \textbf{generated in degree one} if for all $r\in\nn$, $\norm_r$ is the quotient norm induced by $\norm_1^{\odot r}$ via the above morphism $\Phi_r$.
\end{defi}

Bounded graded norms $\norm_\bullet$ generated in degree one are very easy to study: their asymptotic behaviour is heavily controlled by that of $\norm_1$ and of the asymptotic structure of the algebra $R(X,L)$. The main result of this subsection will be a powerful approximation theorem for these norms.

\bsni If $L$ is semiample, there always exists a $r$ such that $R(X,rL)$ is generated in degree one. Thus, while norms generated in degree one may not necessarily exist (as they require $R(X,L)$ to be generated in degree one), the following class is always nonempty:
\begin{defi}
We say that a graded norm $\norm_\bullet$ is \textbf{finitely generated} if there exists a positive integer $r$ such that $\norm_{r\bullet}$ is generated in degree one on $rL$.
\end{defi}

\begin{remark}
We can in fact replace the growth condition in the definition of a bounded graded norm, by defining boundedness as being of finite $d_\infty$-distance with respect to a finitely generated norm.
\end{remark}

\noindent Due to the above remark, it is interesting to see the behaviour of the asymptotic spetral measure restricted to such a subalgebra. We have the following results:
\begin{prop}Recall that, given $f:\rr\to\rr$, and measure $\mu$ on the reals, $f_*\mu$ denotes the pushforward of $\mu$ by $f$. Set $\norm_\bullet$, $\norm'_\bullet\in\mathcal{N}_\bullet(L)$. We then have that:
\begin{itemize}
\item $f(\lambda)=-\lambda\Rightarrow f_*\sigma(\norm_\bullet,\norm'_\bullet)=\sigma(\norm'_\bullet,\norm_\bullet)$;
\item for any $c\in\rr$, $f(\lambda)=\lambda + c\Rightarrow f_*\sigma(\norm_\bullet,\norm'_\bullet)=\sigma(e^{-c}\norm_\bullet,\norm'_\bullet)$;
\item for any $r\in\nn^*$, $f(\lambda)=r\lambda\Rightarrow f_*\sigma(\norm_\bullet,\norm'_\bullet)=\sigma(\norm_{r\bullet},\norm'_{r\bullet})$,
\end{itemize}
where $\norm_{r\bullet}$ denotes the restriction of $\norm_\bullet$ to the subalgebra $R(X,rL)$.
\end{prop}
\noindent This Proposition is the non-trivially valued equivalent of Propositions 3.4 and 3.5 of \cite{ber}.

\bigskip\noindent From now on, in this subsection, we assume $R(X,L)$ to be generated in degree one. We wish to prove an approximation Theorem for norms generated in degree one. We start with some preparatory Propositions and Lemmas.

\begin{prop}[Quotients decrease distance]Let $W\subset V$ be a proper linear subspace of $V$, let $\norm$ and $\norm'$ be norms on $V$. Denote $\tilde{\norm}$ and $\tilde{\norm}'$ the induced quotient norms on $V/W$. We then have that:
$$d_\infty(\tilde{\norm},\tilde{\norm}')\leq d_\infty(\norm,\norm').$$
\end{prop}
\begin{proof}
Fix $a=d_\infty(\norm,\norm')$, and $\tilde{v}\in V/W$. It is enough to show that
$$e^{-a}\tilde{\norm}'(\tilde{v})\leq \tilde{\norm}(\tilde{v})\leq e^a\tilde{\norm}(\tilde{v}).$$
We lift $\tilde{v}$ to a sum $v+w$ with $v\in V-W$ and $w\in W$. Note that
$$e^{-a}\norm'(v+w)\leq \norm(v+w),$$
for all such lifts, so that we can pass to the inf and get that
$$e^{-a}\tilde{\norm}'(\tilde{v})\leq \tilde{\norm}(\tilde{v}).$$
Similarly, we get that
$$e^{-a}\tilde{\norm}(\tilde{v})\leq \tilde{\norm}'(\tilde{v}).$$
The result follows.
\end{proof}

\begin{corollary}\label{gendegonedinfty}Let $\norm_\bullet$, $\norm'_\bullet\in\mathcal{N}_\bullet(L)$ be generated in degree one. We then have that:
$$d_\infty(\norm_\bullet,\norm'_\bullet)= d_\infty(\norm_1,\norm'_1).$$
\end{corollary}
\begin{proof}
This follows on repeatedly applying the previous proposition. Set $a=d_\infty(\norm_1,\norm'_1)$. For any $m>1$, we have that 
$$\phi_m:H^0(L)^{\odot m}\to H^0(mL)$$
is surjective. Consider $s\in H^0(mL)$, and lifts $\tilde{s}$ of $s$ in $H^0(L)^{\odot m}$, which themselves lift to $\tilde{\tilde{s}}\in H^0(L)^{\otimes m}$. We naturally have that
$$e^{-ma}(\norm'_1)^{\otimes m}(\tilde{\tilde{s}})\leq (\norm_1)^{\otimes m}(\tilde{\tilde{s}}) \leq e^{ma} (\norm'_1)^{\otimes m}(\tilde{\tilde{s}}),$$
so that, applying the above Proposition,
$$e^{-ma}(\norm'_1)^{\odot m}({\tilde{s}})\leq (\norm_1)^{\odot m}({\tilde{s}}) \leq e^{ma} (\norm'_1)^{\odot m}({\tilde{s}}),$$
and finally, since a graded norm generated in degree one is a quotient,
$$e^{-ma}(\norm'_m)({s})\leq (\norm_m)(s) \leq e^{ma} (\norm'_m)(s).$$
This establishes
$$d_\infty(\norm_\bullet,\norm'_\bullet)\leq d_\infty(\norm_1,\norm'_1),$$
and since the $d_\infty$ distance is defined as a $\sup$, we in fact have equality.
\end{proof}

\noindent The coming results require specific constructions of bounded graded norms. It will not be possible in general to assume them to be being generated in degree one; however, they will coincide in all high enough degrees with one such norm. Hence, we introduce the following definition, to make our later statements lighter.
\begin{defi}
We say that a bounded graded norm $\norm_\bullet$ is \textbf{eventually generated in degree one} if there exists a norm generated in degree one $\norm^o_\bullet$ on $L$, and a positive integer $r$, such that for all $m\geq r$,
$$\norm_m=\norm^o_m.$$
We will say that $\norm_\bullet$ \textbf{eventually coincides} with $\norm^o_\bullet$.
\end{defi}

\begin{remark}\label{eventuallygdoneremark}
In particular, Corollary \ref{gendegonedinfty} may be reformulated in the context of norms eventually generated in degree one as follows: let $\norm_\bullet$ and $\norm'_\bullet$ be two such norms, eventually coinciding with norms generated in degree one $\norm^o_\bullet$ and $\norm'_\bullet{}^o$ respectively. Then, for all large enough $m$,
$$m\mi d_\infty(\norm_m,\norm'_m)\leq d_\infty(\norm_1^o,\norm'_1{}^o).$$
\end{remark}

\bigskip We now describe how to construct, starting from a lattice norm, a model $(\mathcal{X},\mathcal{L})$ of $(X,L)$, where $L$ is a line bundle whose algebra of multisections is generated in degree one, and satisfying the property that the bounded graded norm associated to the sections $\mathcal{L}$ is eventually generated in degree one.

\bigskip\noindent Consider then $L$ a line bundle on $X$, such that $R(X,L)$ is generated in degree one, and let $\norm$ be a lattice norm on $H^0(L)$, i.e. there exists a basis of sections $(s_i)$ of $H^0(L)$ which is orthonormal for $\norm$. Denote $\mathcal{V}_1$ the $\field^\circ$-submodule of $H^0(L)$ generated by this basis of sections, i.e. the unit ball of $\norm$. Then, the surjective symmetry morphisms $\phi_r: H^0(L)^{\odot r}\to H^0(rL)$ of $R(X,L)$ being surjective for all $r\geq 1$, $\mathcal{V}_1$ induces a $\field^\circ$-subalgebra $\mathcal{V}_\bullet$ of $R(X,L)$, which is furthermore generated in degree one, and torsion-free. The scheme
\begin{equation}\label{projconstruction}
\mathcal{X}=\mathrm{Proj}\,\,\mathcal{V}_\bullet
\end{equation}
is then flat and projective over $\field^\circ$. Let $\mathcal{L}$ be its twisting sheaf $\mathcal{O}_\mathcal{X}(1)$. $(\mathcal{X},\mathcal{L})$ is a model of $(X,L)$. Furthermore, for all $m$ large enough, $H^0(m\mathcal{L})$ coincides with $\mathcal{V}_m$ (see \cite[Ex. II-5.14]{har}). In particular, the sequence of norms
$$(\norm_{H^0(m\mathcal{L})})_m$$
is eventually generated in degree one, and the norm generated in degree one with which it eventually coincides is generated by $\norm$.

\bigskip\noindent We may then prove the following result:
\begin{prop}\label{uniformmodelapproximationprop}Assume $\field$ to be densely valued, and let $\norm_\bullet$ be generated in degree one. Then, for all $\varepsilon>0$, there exists a model $(\mathcal{X}^\varepsilon,\mathcal{L}^\varepsilon)$ of $(X,L)$, such that, for large enough $m$,
$$d_\infty(\norm_m,\norm_{H^0(m\mathcal{L}^\varepsilon)})<m\varepsilon.$$
\end{prop}
\begin{proof}
Since $\field$ is densely valued, for all $\varepsilon>0$, there exists a lattice norm $\norm^\varepsilon$ with
$$d_\infty(\norm_1,\norm^\varepsilon)<\varepsilon.$$
Being a lattice norm, we associate to $\norm^\varepsilon$ a model $(\mathcal{X}^\varepsilon,\mathcal{L}^\varepsilon)$ as in the construction (\ref{projconstruction}) above, whose associated graded norm $\norm_{H^0(\bullet\mathcal{L}^\varepsilon)}$ eventually coincides with the norm generated in degree one by $\norm^\varepsilon$. 

\bigskip\noindent We then use the reformulation in Remark \ref{eventuallygdoneremark} of  Corollary \ref{gendegonedinfty}: since $\norm_\bullet$ and $\norm_{H^0(\bullet\mathcal{L}^\varepsilon)}$ are both eventually generated in degree one, we have that
$$d_\infty(\norm_m,\norm_{H^0(m\mathcal{L}^\varepsilon)})\leq md_\infty(\norm_1,\norm^\varepsilon)<m\varepsilon$$
for all $m$ large enough.
\end{proof}

\noindent Finally, we may prove the main Theorem of this section.

\begin{theorem}\label{uniformmodelapproximation}
Assume $\field$ to be densely valued, and assume $L$ to be such that $R(X,L)$ is generated in degree one. Let $\norm_\bullet$, $\norm'_\bullet$ be bounded graded norms generated in degree one. Then, for all $\varepsilon>0$, there exist models $(\mathcal{X}^\varepsilon,\mathcal{L}^\varepsilon)$ and $(\mathcal{Y}^\varepsilon,\mathcal{M}^\varepsilon)$ of $(X,L)$, such that:
$$\vol(\norm_{H^0(\bullet\mathcal{L^\varepsilon})},\norm_{H^0(\bullet\mathcal{M}^\varepsilon)})\to_{\varepsilon\to0} \vol(\norm_\bullet,\norm'_\bullet).$$
\end{theorem}
\begin{proof}
We pick sequences of models $(\mathcal{X}^\varepsilon,\mathcal{L}^\varepsilon)$ and $(\mathcal{Y}{}^\varepsilon,\mathcal{M}^\varepsilon)$ of $(X,L)$ as in Proposition \ref{uniformmodelapproximationprop}. Using the cocycle condition on volumes, we have that
$$\vol(\norm_\bullet,\norm'_\bullet)=\vol(\norm_\bullet,\norm_{H^0(\bullet\mathcal{L}^\varepsilon)})+\vol(\norm_{H^0(\bullet\mathcal{L}^\varepsilon)},\norm_{H^0(\bullet\mathcal{M}^\varepsilon)})+\vol(\norm_{H^0(\bullet\mathcal{M}^\varepsilon)},\norm'_\bullet),$$
so that it is then enough to prove
$$\vol(\norm_\bullet,\norm_{H^0(\bullet\mathcal{L}^\varepsilon)})\to_{\varepsilon\to 0} 0.$$
(The proof for $\mathcal{L}$ being also valid for $\mathcal{M}$.) Since volumes respect a Lipschitz property with respect to the $d_\infty$-distance (Proposition \ref{volumelipschitz}), we have
$$|\vol(\norm_\bullet,\norm_{H^0(\bullet\mathcal{L}^\varepsilon)})|=|\vol(\norm_\bullet,\norm_{H^0(\bullet\mathcal{L}^\varepsilon)}) - \vol(\norm_\bullet,\norm_\bullet)|\leq \limsup_m m\mi d_\infty(\norm_m,\norm_{H^0(m\mathcal{L}^\varepsilon)}).$$
In light of Proposition \ref{uniformmodelapproximationprop}, we then have that
$$|\vol(\norm_\bullet,\norm_{H^0(\bullet\mathcal{L}^\varepsilon)})|\leq d_\infty(\norm_1,\norm_{H^0(\mathcal{L})})< \varepsilon,$$
which concludes the proof.
\end{proof}

\section{Chebyshev transforms of graded norms.}

The final result of this section relies on volumes, which are invariant under ground field extension. We may then consider an algebraically closed, non-trivially valued field $\field$. The value group of $\field$ is then divisible, hence dense. As in the previous section, $L$ is assumed to be a semiample line bundle over a $\field$-variety $X$.

\subsection{Okounkov bodies.}\label{idk}

Let $x$ be a regular $\field$-rational point of $X$, and pick a regular sequence $(z_1,\dots,z_d)$ in the local ring $\oo_{X,x}$. By Cohen's structure theorem, any element $f\in\oo_{X,x}$ may then be written as a formal power series
$$f=\sum_{\alpha\in\mathbb{N}^d}f_\alpha z^\alpha,$$
where the coefficients $f_\alpha$ belong to the field $\field$. 
Given a monomial order $\leq$ on $\mathbb{Z}^d$, we then define a valuation
\begin{align*}
\ord_{x,\leq}:K(X)^*&\to \mathbb{Z}^d\\
f=\sum_{\alpha\in\mathbb{N}^d}f_\alpha z^\alpha&\mapsto \ord_{x,\leq}(f)=\min{}_\leq\{\alpha\in\mathbb{N}^d,\,f_\alpha\neq 0\}.
\end{align*}
Given a section $s\in H^0(X,L)$, one may pick a trivialization of $L$ at $x$, so that $s$ defines an element $s_x\in \oo_{X,x}$, and we extend the valuation $\ord$ to sections of $L$ by setting
$$\ord_{x,\leq}(s)=\ord_{x,\leq}(s_x).$$
Note that this is independent of the choice of a trivialization. We then set
$$\gr_{n,\alpha}(L)=\{s\in H^0(mL),\,\ord_{x,\leq}(s)\geq \alpha\}/\{s\in H^0(mL),\,\ord_{x,\leq}(s)> \alpha\},$$
and notice that
$$\dim_\field \gr_{n,\alpha}(L)\leq1$$
for all choices of $(n,\alpha)\in\nn^{d+1}$. We then consider the sub-semigroup $\Gamma_m(L)\subseteq\mathbb{N}^d$ consisting of the values taken by $\ord_{x,\leq}$ evaluated on the space of sections $H^0(mL)$, that is,
$$\Gamma_m(L)=\ord_{x,\leq}(H^0(mL)),$$
and the graded semigroup $\Gamma(L)\subseteq \mathbb{N}^{d+1}$ defined as the union of all the $\Gamma_m(L)$, with a $\mathbb{N}$-grading given by the tensor power of $L$ considered:
$$\Gamma(L)=\{(n,\alpha)\in\mathbb{N}\times\mathbb{N}^d,\,\alpha\in\Gamma_n(L)\}.$$
It then follows that $(n,\alpha)\in\Gamma(L)$ if and only if $\dim \gr_{n,\alpha}(L)=1$. The semigroup $\Gamma(L)$ then satisfies the following properties:
\begin{itemize}
\item[(i)] \textit{linear growth:} $\Gamma(L)$ is contained within a finitely generated sub-monoid $\langle a_1,\dots,a_k\rangle\subseteq \{1\}\times\mathbb{N}^d$, $k<\infty$;
\item[(ii)] \textit{bigness:} $\Gamma(L)$ generates $\mathbb{N}^{d+1}$ as a group. 
\end{itemize}
See \cite[L2.11, P3.3]{boubbk} for a precise justification of these facts. It follows that the base
$$\Delta(\Gamma(L))=\overline{\mathrm{Cone}}(\Gamma(L))\cup(\{1\}\times\rr^d)$$
of the convex cone generated by $\Gamma(L)$ inside $\rr^{d+1}$ defines a convex body, that is, a subset of $\rr^d$ satisfying the following properties:
\begin{itemize}
\item $\Delta(\Gamma(L))$ is compact and convex;
\item the interior $\Delta(\Gamma(L))^o$ is nonempty.
\end{itemize}
In fact, we may associate such a convex body $\Delta(\Gamma)$ to any semigroup $\Gamma\subseteq \mathbb{N}^{d+1}$ satisfying the properties $(i)-(ii)$ above. (See also \cite[2.A]{macleanroe} for a definition of Okounkov bodies with minimal conditions on $\Gamma$.) We call $\Delta(\Gamma)$ the \textbf{Okounkov body} of the semigroup $\Gamma$.

\subsection{The asymptotic spectral measure.}

\noindent Consider as before a monomial order $\leq$ on $\nn^d$, a regular rational point $x\in X$, and a regular system of parameters. Let $(n,\alpha)$ belong to $\Gamma(L)$, which as we recall holds if and only if the graded piece $\gr_{n,\alpha}$ to be one-dimensional. Hence, given a trivialization $\tau_x$ of $L$ at $x$, there exist sections $s\in H^0(nL)$ with Taylor expansion
$$s=z^\alpha + \sum_{\beta\geq \alpha}s_\beta z^\beta$$
with respect to $\tau_x^n$. Given a bounded graded norm $\norm_\bullet$, the individual norm $\norm_n$ induces a quotient norm $\norm_{n,\alpha}$ on $\gr_{n,\alpha}(L)$. Given any $s$ with a Taylor expansion as above, it is immediate that its class $[s]_{n,\alpha}=:s_{n,\alpha}$ in $\gr_{n,\alpha}(L)$ contains all such sections, and we define
\begin{align*}
\Phi:\Gamma(L)&\to \rr\\
n,\alpha&\mapsto -\log\left[\,\norm_{n,\alpha}\left(s_{n,\alpha}\right)\,\right].
\end{align*}
By submultiplicativity of $\norm_\bullet$ and the fact that $$s_{n,\alpha}\cdot s_{m,\beta}=s_{n+m,\alpha+\beta}$$
in the algebra
$$\bigoplus_{n\in\mathbb{N}}\bigoplus_{\alpha\in\Gamma_k(L)}\gr_{n,\alpha}(L),$$
the function $\Phi$ so defined is then superadditive.

\bigskip

\noindent We now prove existence of the limit spectral measure between two bounded graded norms. This Theorem has originally been proved in \cite[T5.2]{cmac}, albeit in a slightly different mathematical language, and some parts have been further developed in \cite[T9.5]{boueri}. We write out the proof in "our language" below.

\begin{proof}By \cite[T5.2]{cmac}, such a limit measure exists provided the hypotheses of \cite[T4.5]{cmac} are verified. We state them now:
\begin{itemize}
\item[(1)] $\norm_\bullet$ and $\norm'_\bullet$ are submultiplicative graded norms;
\item[(2)] $\lim_{m\to\infty} m\mi d_\infty(\norm_m,\norm'_m)<\infty$;
\item[(3)] there exists uniform positive constant $C$ such that
$$\inf_{\alpha\in\Gamma_n(L)}\log\norm_{n,\alpha}(s_{n,\alpha})\geq -Cn,$$
and similarly for $\norm'_\bullet$.
\end{itemize}

\bigskip

\noindent The two first criteria $(1)$ and $(2)$ are by definition true since $\norm_\bullet$ and $\norm'_\bullet$ are bounded graded norms. What remains is to prove $(3)$ which is equivalent to showing that $(n,\alpha)\mapsto\Phi(n,\alpha)$ is \textit{linearly bounded above} in the first variable, i.e. there exists a uniform positive constant $C$ such that
$$\Phi(n,\cdot)\leq Cn.$$
Since the (quotient) norm $\norm_{n,\alpha}$ is characterized as an $\inf$ on all sections $s\in H^0(nL)$ with Taylor expansion at $x$ of the form
\begin{equation}\label{taylor}s=z^\alpha + \sum_{\beta\geq \alpha}s_\beta z^\beta,
\end{equation}
we have that, if
$$-\log \norm_n(s)\leq Cn$$
is true for all such $s$, then this is also true for the $\inf$ on all such $s$, $\Phi(n,\alpha)=-\log\norm_{n,\alpha}(s_{n,\alpha})$. In fact, by the finite growth property of $\Gamma(L)$, we know that there exists a uniform positive constant $C'$ such that
\begin{equation}\label{boundedgrowthsemigroup}\alpha\in\Gamma_n(L)\Rightarrow |\alpha|\leq C'n.
\end{equation}
Thus, it is enough to show that
\begin{equation}\label{ineq}-\log \norm_n(s)\leq C(n+|\alpha|)
\end{equation}
for all $s$ with Taylor expansion as in (\ref{taylor}). We then follow the proof of \cite[T9.5]{boueri}, itself based on \cite[L5.4]{dwn}. We first reduce to the case where $\norm_\bullet$ is the supnorm $\norm_{\bullet\phi}$ associated to a bounded metric $\phi$, for $\norm_\bullet$ is bounded. This implies that for some bounded metric $\phi$,
$$\lim_m m\mi d_\infty(\norm_m,\norm_{m\phi})=D<\infty$$
so that if (\ref{ineq}) is true for $\norm_{\bullet\phi}$ then, as for all large enough $n$ we have
$$e^{-Dn}\norm_{n\phi}\leq \norm_n \leq e^{Dn}\norm_{n\phi},$$
it follows that
$$Dn - \log\norm_{n\phi}(s) \geq -\log \norm_n(s) \geq -Dn - \log\norm_{n\phi}(s),$$
and the equivalent of (\ref{ineq}) for $\norm_\bullet$ is deduced from this, and from (\ref{ineq}).

\bigskip

\noindent We then assume that $\norm_\bullet=\norm_{\bullet\phi}$ for some bounded metric $\phi$ on $L$. Now, we know that we can find a trivialization $\tau_x$ of $L$ and analytic isomorphisms from a neighborhood $U$ of a regular rational point $x\in X$ to an open polydisc $\mathbb{D}=\prod_1^d\mathbb{D}(r_i)\subset \field^d$, such that a section $s\in H^0(nL)$ satisfies
$$\log |s|_{n\phi}=\log |s_U| + n\log|\tau_x|_\phi,$$
for some analytic function $s_U$ of the form
$$f(z)=z^\alpha + \sum_{\beta\geq \alpha}s_\beta z^\beta.$$
Since $\phi$ is bounded on $U$, so is the term $n\log|\tau_x|_\phi$, and by the maximum principle, applied in each variable, we have that
$$r^{|\alpha|}\leq |f|$$
on $U$, and finally (\ref{ineq}) follows, concluding the proof of the Theorem.
\end{proof}

\subsection{A Fujita approximation lemma.}

\noindent In this subsection, we prove a result which is similar in spirit to Fujita's approximation Theorem. We first start by quoting the following standard result:

\begin{theorem}[{\cite[L1.13]{boubbk}}]\label{theoremcompactokounkov}
Let $\Gamma$ be a graded sub-semigroup of $\nn^{d+1}$ satisfying conditions (i)-(ii) of \ref{idk}, and let $K$ be a compact convex subset of $\rr^d$ contained in the interior of $\Delta(\Gamma_\bullet)$. For all large enough integers $m$, we then have that:
$$K\cap \frac{\Gamma_m}{m} = K \cap \frac{\mathbb{Z}^d}{m},$$
where $\Gamma_m$ is defined again as in \ref{idk}.
\end{theorem}

\noindent We now prove said approximation result, which can also be seen as an \textit{ad hoc} version of \cite[L1.21]{boubbk}.

\begin{lemma}\label{fujita}Let $\Gamma^k$ be a graded sub-semigroup of some semigroup $\Gamma\subseteq \rr^{d+1}$, such that:
\begin{itemize}
\item $\Gamma_1^k=\Gamma_k$,
\item $\Gamma_r^k\subseteq \Gamma_{kr}$ for all $r\geq 1$,
\item $\Gamma$ is admissible.
\end{itemize}
Then
$$k^{-d}\vol(\Delta(\Gamma_\bullet^k))\to_{k\to\infty}\vol(\Delta(\Gamma_\bullet)).$$
\end{lemma}
\begin{proof}First remark that, by the inclusion property
$$\Gamma_r^k\subseteq \Gamma_{kr},$$
we have that, for all $k\geq 1$,
$$\frac{\Delta(\Gamma^k_\bullet)}{k}\subseteq\Delta(\Gamma_\bullet).$$
If we can show that any compact (convex) subset $K$ of $\Delta(\Gamma_\bullet)^o$ is also included in $\frac{\Delta(\Gamma^k_\bullet)}{k}$ for large enough $k$, then our assertion would be true. Pick such a compact $K$, and embed it into another compact convex subset $L\subset\Delta(\Gamma_\bullet)^o$ such that the number
$$d(K,\partial L)=\inf \,\{d(x,\ell),\,x\in K,\,\ell\in \partial L\}$$
is (strictly) positive. We then have compact inclusions
$$K\subset L \subset \Delta(\Gamma_\bullet)^o,$$
with $K$ not "touching" the boundary of $L$.

\noindent By admissibility, $\Gamma_\bullet$ generates $\mathbb{Z}^{d+1}$ as a group. Then, the regularization of $\Gamma_k$ is $\mathbb{Z}^d$, whence, for all large enough $k$,
$$\left(L\cap \frac{\Gamma_k}{k}\right) = \left(L\cap \frac{\mathbb{Z}^d}{k}\right),$$
(by Theorem \ref{theoremcompactokounkov}), so that the convex hull of $\left(\frac{\Gamma_k}{k}\right)$ naturally contains $K$. (It does not necessarily contain $L$.) Now, since
$$\Gamma_1^k=\Gamma_k,$$
the convex hull of $\left(\frac{\Gamma_k}{k}\right)$ is contained in the scaled Okounkov body
$$\frac{\Delta(\Gamma^k_\bullet)}{k}.$$
To conclude, we have a chain of compact inclusions
$$K\subset \mathrm{Hull}\left(\frac{\Gamma_k}{k}\right) \subset \frac{\Delta(\Gamma^k_\bullet)}{k},$$
from which follows the desired inclusion of $K$.
\end{proof}

\noindent In practice, we will consider $\Gamma=\Gamma(L)$, the Okounkov semigroup associated to the choice of a semiample line bundle $L$ over $X$, a monomial order $\leq$ on $\nn^d$, and a regular rational point $x\in X$.

\subsection{Chebyshev functions associated to superadditive functions.}\label{superaddoko}

\begin{prop}[{\cite[L4.1, T4.3]{cmac}}]Assume $\Phi$ is a superadditive function
$$\Phi:\Gamma(L)\to \rr,$$
such that $\Phi(0,0_{\nn^d})=0$. For any $t\in\rr$, set
$$\Gamma^{\geq t}_\Phi=\{(n,\alpha)\in \Gamma(L),\,\Phi(n,\alpha)\geq n\cdot t\}.$$
Then, $\Gamma^{\Phi,\geq t}$ sub-semigroup of $\Gamma(L)$ satisfying properties (i)-(ii) of \ref{idk} whenever
$$t<\theta=\lim_{n\to\infty}\sup_{\alpha\in\Gamma_n(L)}n\mi \Phi(n,\alpha).$$
\end{prop}

\begin{remark}It is immediate that
$$l<t\Rightarrow \Gamma^{\Phi,\geq l}\subseteq \Gamma^{\Phi,\geq t}.$$
\end{remark}

\begin{defi}
Let $\Phi$ be a superadditive function on $\Gamma(L)$. We set
$$G_\Phi:\Delta(\Gamma(L))\to \rr\cup \{-\infty\},$$
$$(n,\alpha)\mapsto \sup \{t\in\rr\cup\{-\infty\},\,(n,\alpha)\in \Delta(\Gamma^{\Phi,\geq t})\}.$$
The function $G_\Phi$ is the \textbf{Chebyshev function} of the semigroup $\Gamma(L)$ (associated to $\Phi$). The term \textit{concave transform} is also common in the literature, see e.g. \cite{dwn} and \cite{macleanroe}.
\end{defi}

\begin{remark}By \cite{boucksomchen} this function is concave, hence continuous, on the interior of $\Delta(\Gamma(L))$.
\end{remark}

\subsection{An equidistribution result.}

\noindent We rely on a first result taken from \cite{cmac}, which can be seen as an equidistribution Theorem for the values of a superadditive function defined on an Okounkov body. We in particular obtain a limit measure, and a sequence of such limit measures is the main character of a Proposition proven below.
\begin{theorem}[{\cite[T4.3, R4.4]{cmac}}]\label{integralvolume}Let $\Phi$ be a superadditive function from $\Gamma(L)$ to $\rr$, with $\limsup$ denoted $\theta$ as in the previous subsection. Let $\mu(k)$ be the finitely supported probability measure on $\rr$ defined as
$$\mu(k)=\sum_{\alpha\in\Gamma_k(L)}\delta_{k\mi \Phi(k,\alpha)}.$$
This sequence then converges to a compactly supported probability measure $\mu$ on $\rr$ satisfying
$$\mu([t,\infty))=\frac{\vol(\Delta(\Gamma^{\Phi,\geq t}_\bullet))}{\vol(\Delta(\Gamma_\bullet))},$$
for any $t\leq \theta$. Furthermore, $\mu$ is equal to the pushforward of the normalized Lebesgue measure on the Okounkov body $\Delta(\Gamma(L))$ by the Chebyshev function $G_\Phi$.
\end{theorem}

\begin{defi}
If $\Phi$ is a superadditive function defined on an Okounkov body, associated to a bounded graded norm $\norm_\bullet$ as before, we denote the limit measure obtained in the previous Theorem by
$$\mu(\norm_\bullet),$$
and the measures  $\mu(k)$ as
$$\mu(\norm_k).$$
\end{defi}

We now state the main technical result of this section.

\begin{prop}\label{mainthmsectionfour}
Let $L$ be such that $R(X,L)$ is generated in degree one. Let $\norm_\bullet$ be a bounded graded norm on $R(X,L)$. Consider, for each $k\in\mathbb{N}^*$, the bounded graded norm $\norm^{(k)}_\bullet$ on $R(X,kL)$ generated in degree one by $\norm_k$, i.e. the sequence of quotient norms induced by $\norm_k$ and the symmetry morphisms
$$H^0(kL)^{\odot r}\twoheadrightarrow H^0(rkL)$$
for all $r\in\mathbb{N}^*$. Set
$$\Gamma(kL)=\{(n,\alpha)\in \Gamma(L),\,k|n\}.$$
We then have 
$$\mu(\norm_\bullet^{(k)})\rightharpoonup_{k\to\infty}\mu(\norm_\bullet),$$
where $\rightharpoonup$ denotes weak convergence of measures, in particular: the sequence of functions $t\mapsto \int_{-\infty}^t\,d\mu(\norm^{(k)}_\bullet)$ converges pointwise to $t\mapsto \int_{-\infty}^t\,d\mu(\norm_\bullet)$.
\end{prop}
\begin{proof}
\noindent Denote $\Phi$ and $\Phi_k$ be the superadditive functions associated to the norms $\norm_\bullet$ and $\norm^{(k)}_\bullet$. We first notice the following properties of $\Phi$ and the $\Phi_k$:
\begin{itemize}
\item[(i)] $\Phi_k(k,\alpha)=\Phi(k,\alpha)$, for all $(k,\alpha)\in \Gamma_k$;
\item[(ii)] $\Phi_k(kn,\alpha)\leq \Phi(kn,\alpha)$, for all $(kn,\alpha)\in \Gamma_{k\bullet}$;
\item[(iii)] if $d|k$, then $\Phi_d(kn,\alpha)\leq\Phi_{k}(kn,\alpha)$, for all $(kn,\alpha)\in \Gamma_{k\bullet}$.
\end{itemize}

\noindent Let $\theta_k$ and $\theta$ be the above bounds on the supports of the appropriate measures. We then show that
$$\mu(\norm^{(k)}_\bullet)([t,\theta])\to \mu(\norm_\bullet)([t,\theta]),$$
for all $t\in [-\infty,\theta]$.

\noindent Now, since
$$\mu(\norm^{(k)}_\bullet)([t,\theta])=\frac{\vol\left(\Delta(\Gamma_{k\bullet}^{\Phi_k,\geq t})\right)}{\vol(\Delta(\Gamma_{k\bullet}))},$$
and
$$\vol(\Delta(\Gamma_{k\bullet}))\mi=\vol(\Delta(\Gamma_{\bullet}))\mi,$$
the problem reduces to showing that the sequence of functions $(v_k)_k$, defined as
$$v_k:t\mapsto \vol\left(\Delta(\Gamma_{k\bullet}^{\Phi_k,\geq t})\right)$$
converges pointwise to
$$v:t\mapsto \vol(\Delta(\Gamma^{\Phi,\geq t}_\bullet)).$$
Note that (ii), (iii), and the expressions
$$\theta=\lim_{n\to\infty}\sup_{\Gamma_n}\frac{\Phi(n,\alpha)}{n}<\infty,$$
and
$$\theta_k=\lim_{n\to\infty}\sup_{(kn,\alpha)\in\Gamma_{kn}}\frac{\Phi_k(kn,\alpha)}{n}<\infty$$
imply that $(\theta_{k})_k$ is an increasing sequence converging to $\theta$. 

\noindent Finally, the semigroups $\Gamma_{(k)\bullet}^{\Phi_{(k!)},\geq t}$ and $\Gamma^{\Phi,\geq t}_\bullet$, satisfy the hypotheses of Lemma \ref{fujita} (note (i)), which yields
$$v_k(t)\to v(t),$$
concluding the proof.
\end{proof}

\noindent This implies the main Theorem of this section, Theorem C.
\begin{theorem}\label{thecorollarywemustuse}Let $L$ be such that $R(X,L)$ is generated in degree one. Let $\norm_\bullet$, $\norm'_\bullet$ be two bounded graded norms on $L$, and for each $k\in\nn^*$, let $\norm_\bullet^{(k)}$ and $\norm'_\bullet{}^{(k)}$ denote the graded norms generated in degree one by $\norm_k$ and $\norm'_k$ respectively. Then, we have that:
$$\vol(\norm_\bullet^{(k)},\norm'_\bullet{}^{(k)})\to_{k\to\infty}\vol(\norm_\bullet,\norm'_\bullet).$$
\end{theorem}
\begin{proof}
Let $\Phi'$ and for all $k$, $\Phi'_k$ be the superadditive functions associated to the norms $\norm'_\bullet$ and $\norm'{}^{(k)}_\bullet)$ respectively.

\bigskip

\noindent Recall the identity
$$\vol(\norm_\bullet,\norm'_\bullet)=\lim_{m\to\infty}m\mi\vol(\norm_\bullet,\norm'_\bullet).$$
Note that
$$\int_\rr \lambda\,d\mu(\norm_m) - \int_\rr \lambda\,d\mu(\norm'_m)=m\mi\sum_{\alpha\in\Gamma_m(L)}\left[\Phi(m,\alpha)-\Phi'(m,\alpha)\right],$$
where $\mu(\norm_m)$ and $\mu(\norm'_m)$ are defined as the finitely supported measures as in Theorem \ref{integralvolume}. By \cite[(29)]{cmac}, the quantity on the right is identified with
$$m\mi\vol(\norm_m,\norm'_m),$$
so that at the limit,
$$\int_\rr \lambda\,d\mu(\norm_\bullet) - \int_\rr \lambda\,d\mu(\norm'_\bullet)=\vol(\norm_\bullet,\norm'_\bullet).$$
Doing the same process with $\norm^{(k)}_\bullet$ and $\norm'{}^{(k)}_\bullet$, we then find that
$$\int_\rr \lambda\,d\mu(\norm^{(k)}_\bullet) - \int_\rr \lambda\,d\mu(\norm'{}^{(k)}_\bullet)=\vol(\norm_\bullet^{(k)},\norm_\bullet'{}^{(k)}).$$
An application of Theorem \ref{mainthmsectionfour} then yields the desired convergence.
\end{proof}

\begin{remark}\label{reduction}
As volumes are invariant under ground field extension, and in view of the properties of the limit measure under pushforward, the results remain true whenever $L$ is a semiample $\qq$-line bundle, and $\mathbb{K}$ is any non-Archimedean field.
\end{remark}

\section{The asymptotic Fubini-Study operator.}

We at first assume $\field$ to be any non-trivially valued, non-Archimedean field.

\subsection{Volumes and energies.}

The goal of this section is to prove the following Theorem, a generalization of \cite[T4.13]{ber}, where we consider a general complete non-Archimedean field, rather than one which is trivially valued.

\begin{theorem}[Theorem B]\label{theorema}
Let $\norm_\bullet$, $\norm'_\bullet\in\mathcal{N}_\bullet(L)$. We then have:
$$\lim_m E(\FS_m(\norm_m),\FS_m(\norm'_m))=\vol(\norm_\bullet,\norm'_\bullet).$$
\end{theorem}

\noindent As a first reduction, we can assume $L$ to be globally generated and the algebra of sections to be generated in degree one (thanks to Remark \ref{reduction}).

\bigskip

\noindent We now show that we can reduce to the case where $\mathbb{K}$ is algebraically closed and non-trivially valued (hence densely valued).

\begin{lemma}Assume $\field$ to be any non-trivially valued, non-Archimedean field, and that Theorem B holds for the base change of $X$ to an algebraically closed extension $\mathbb{L}$ of $\field$. Then, Theorem A holds for $X/\field$.
\end{lemma}
\begin{proof}
By Remark \ref{reduction}, the right-hand side is indeed invariant under ground field extension, so that we only have to take care of the energy side of the equation. Consider the base change $X_{\mathbb{L}}$ and its pullback line bundle $L_{\mathbb{L}}$. Note that the ground field extension $R(X,L)_{\mathbb{L}}$ of the algebra of sections of $L$ coincides with $R(X_{\mathbb{L}},L_{\mathbb{L}})$. Consider the associated norms $\norm_{\bullet,\mathbb{L}}$ and $\norm'_{\bullet,\mathbb{L}}$.
\begin{itemize}
\item by Proposition \ref{extvol}, the Fubini-Study operators associated to each individual norm coincide with those associated to their ground field extension, and that (say)
$$\FS_m(\norm_{m,\mathbb{L}})={\pi_1}^*\FS_m(\norm_m);$$
\item by Proposition \ref{extma}, $${\pi_1}_*\MA(\FS_m(\norm_{m,\mathbb{L}}),\FS_m(\norm'_{m,\mathbb{L}}))=\MA(\FS_m(\norm_m),\FS'_m(\norm_m)),$$ where $\MA(\phi,\phi')$ denotes any mixed Monge-Ampère measure involving only $\phi$ and $\phi'$.
\end{itemize}
It follows that both quantities in the assertion of Theorem B are invariant under ground field extension. Using that the Theorem then holds over $X_{\mathbb{L}}$, this finishes the proof.
\end{proof}

\noindent From now on, assume $\field$ to be algebraically closed and non-trivially valued. We first consider the basic case.

\begin{lemma}Theorem B is true when $\norm_\bullet$ and $\norm'_\bullet$ are both graded norms generated in degree one.
\end{lemma}
\begin{proof}
Pick approximations $\norm_{H^0(\bullet\mathcal{L}^\varepsilon)}$ and $\norm_{H^0(\bullet\mathcal{M}^\varepsilon)}$ as in Theorem \ref{uniformmodelapproximation}. By Lemma \ref{basecase} below, we have, for all $\varepsilon>0$,
$$E(\phi_{\mathcal{L}^\varepsilon},\phi_{\mathcal{M}^\varepsilon})=\vol(\norm_{H^0(\bullet\mathcal{L}^\varepsilon)},\norm_{H^0(\bullet\mathcal{M}^\varepsilon)}).$$
Now, the statement of Theorem \ref{uniformmodelapproximation} is that
$$\vol(\norm_{H^0(\bullet\mathcal{L}^\varepsilon)},\norm_{H^0(\bullet\mathcal{M}^\varepsilon)})\to_{\varepsilon\to0}\vol(\norm_\bullet,\norm'_\bullet).$$
In particular, by construction, we have that
$$\lim_m \FS_m(\norm_{H^0(m\mathcal{L}^\varepsilon)})=\FS_1(\norm_{H^0(\mathcal{L}^\varepsilon)})=\phi_{\mathcal{L}^\varepsilon},$$
so that the Lemma is proven once we show that
$$\lim_m E(\FS_m(\norm_{H^0(m\mathcal{L}^\varepsilon)}),\FS_m(\norm_{H^0(m\mathcal{M}^\varepsilon)}))\to_{\varepsilon\to0}\lim_m E(\FS_m(\norm_m),\FS_m(\norm'_m)).$$
But using the $1$-Lipschitz property of the operator $\FS_m$ with respect to the $\sup$ norm of metrics and the $d_\infty$-distance, we find that for all $m$, for all $\varepsilon>0$,
$$\sup_X|\FS_m(\norm_{H^0(m\mathcal{L}^\varepsilon)})- \FS_m(\norm_m)|\leq \varepsilon,$$
so that finally, 
$$\lim_m\FS_m(\norm_{H^0(m\mathcal{L}^\varepsilon)})\to_{\varepsilon\to 0} \lim_m\FS_m(\norm_m)|,$$
uniformly. Proceeding similarly for $\mathcal{M}$, and then using continuity of the Monge-Ampère energy along uniform limits, we find the desired result.
\end{proof}

\noindent We then have the following lemma, as promised.
\begin{lemma}\label{basecase}
Assume $(\mathcal{X},\mathcal{L})$, $(\mathcal{Y},\mathcal{M})$ to be semiample models of $L$ defined on the same model $\mathcal{X}$ of $X$.  Denoting $\phi_\mathcal{L}$ and $\phi_\mathcal{M}$ their associated model metrics, we then have that
$$E(\phi_{\mathcal{L}},\phi_{\mathcal{M}})=\vol(\norm_{H^0(\bullet\mathcal{L})},\norm_{H^0(\bullet\mathcal{M})}).$$
\end{lemma}
\begin{proof}
We first start by stating the following equality (\cite[L9.17]{boueri}):
$$\vol(\norm_{H^0(\bullet\mathcal{L})},\norm_{H^0(\bullet\mathcal{M})})=\vol(\bigN_\bullet(\phi_\mathcal{L}),\bigN_\bullet(\phi_\mathcal{M})).$$
Recall from Remarks \ref{remarkvolumes} and \ref{remarkenergies} that our conventions for the volume and energy are different from those of \cite{boueri}, but as semiampleness of $L$ implies
$$\lim \frac{h^0(mL)}{m^{\dim X}}=\vol(L),$$
the changes cancel out. Furthermore, their notation
$$\vol(L,\phi,\psi)$$
corresponds to
$$\vol(\bigN_\bullet(\phi),\bigN_\bullet(\psi))$$
in our case.

\bsni The above equality follows from earlier results of \cite{boueri}, wherein it is shown that
$$d_\infty(\norm_{H^0(m\mathcal{L})},\bigN_m(\phi_\mathcal{L}))=O(1),$$
(\cite[T6.4]{boueri}) so that Lipschitz continuity of the volume with respect to the $d_\infty$-distance concludes. Then, the Lemma is proven by applying Theorem 9.15 of \cite{boueri} to $\phi_\mathcal{L}$ and $\phi_{\mathcal{M}}$.
\end{proof}

We now prove the Theorem.
\begin{proof}
Assume now that both norms are not necessarily finitely generated. By surjectivity of $H^0(kL)^{\odot m}\to H^0(kmL)$ for all $k$, $m>0$, we may endow each $H^0(kmL)$ with the quotient norm induced by this morphism using the norms $\norm_k$, $\norm'_k$. We denote these norms $\norm^{(k)}_m$, $\norm'^{(k)}_m$. These define graded norms, generated in degree one, on $R(k)$. Consider their associated Fubini-Study metrics:
$$\FS_k(\norm^{(k)}_\bullet)$$
and
$$\FS_k(\norm'^{(k)}_\bullet).$$ 
Recall that the Theorem holds for those norms. Now, since $( \FS_k(\norm^{(k)}_\bullet) )_k$, resp. $(\FS_k({\norm'}^{(k)}_\bullet) )_k$ are decreasing nets, by continuity of $E$ along decreasing nets follows:
$$\lim_{k\to \infty} E\left(\FS_k(\norm^{(k)}_\bullet), \FS_k(\norm'^{(k)}_\bullet)\right)=E(\lim_k \FS_k(\norm_k),\FS_k(\norm'_k)).$$
The right-hand side limit, that is,
$$\lim_{k\to\infty}\vol(\norm_\bullet^{(k)},{\norm'}_\bullet^{(k)})=\vol(\norm_\bullet,\norm'_\bullet),$$
is the statement of Theorem \ref{thecorollarywemustuse}.
\end{proof}

\subsection{Envelopes.}

Following \cite[7.5]{boueri}, we define two important notions of envelopes for bounded functions.

\begin{defi}Let $\phi$ be a bounded metric on $L$. The \textbf{\textit{psh} envelope} of $\phi$ is defined as
$$P(\phi)=\sup\{\phi'\in \PSH(L),\,\phi'\leq \phi\}.$$
The \textbf{regular psh envelope} of $f$ is defined as
$$Q(\phi)=\sup\{\phi'\in \PSH(X)\cap C^0(X),\,\phi'\leq \phi\}.$$
We similarly define those envelopes for functions, by identifying the space of continuous metrics on a line bundle as an affine space modelled on the space of continuous functions as before.
\end{defi}

\begin{defi}[Continuity of envelopes] We say that the pair $(X,L)$ admits \textbf{continuity of envelopes} if the following property holds true:
\begin{itemize}
\item if $\phi$ is a continuous metric on $L$, then $P(\phi)$ is continuous.
\end{itemize}
Note that we then have $P(\phi)=Q(\phi)$.
\end{defi}

\begin{example}\label{cpe}By \cite{boulb}, a smooth, projective variety $X$ defined over any field $\field$ which satisfies all of the following properties:
\begin{itemize}
\item $\field$ is of equal characteristic $0$;
\item $\field$ is either trivially or discretely valued,
\end{itemize}
admits continuity of envelopes for any line bundle $L$ over $X$.
Furthermore, by \cite{gub}, continuity of envelopes also holds:
\begin{itemize}
\item for any line bundle on a curve, over any field; 
\item for all line bundles on a variety $X$ over $\field$, where $\field$ is a discretely valued field of positive characteristic $p$, \textit{provided we have resolution of singularities} over $\field$ in any dimension.
\end{itemize}
\end{example}

\begin{defi}
The \textbf{asymptotic Fubini-Study operator} is defined on the set of graded norms as the \textit{upper semi-continuous regularization} of the limit of the usual Fubini-Study operators, that is, given a bounded graded norm $\norm_\bullet$, 
$$\FS(\norm_\bullet)=\mathrm{usc}\left(\lim_m \FS_m(\norm_m)\right).$$
It associates a bounded psh function to any graded norm.
\end{defi}

\begin{remark}The asymptotic Fubini-Study operator is well-defined and defines a psh function provided that $(X,L)$ admits continuity of envelopes, due to \cite[L7.30]{boueri}, as, by Fekete's lemma,
$$\lim_m \FS_m(\norm_m)=\sup_m \FS_m(\norm_m).$$
\end{remark}

\subsection{Plurisubharmonic functions regularizable from below.}

In this section, we investigate the image of the asymptotic Fubini-Study operator. Most of the material here is simply the translation of results from \cite{ber} from the trivially valued case to our case.

\bsni We first need to consider an important set of valuations contained in the Berkovich analytification of a variety. 
\begin{defi}
A valuation $\nu:K(X)^*\to\rr$ is said to be \textbf{divisorial} if there exists a normal, projective, birational model $Y\to X$ of $X$, a prime divisor $E\subset Y$, and a positive real number $c$, such that
$$\nu=c\cdot \mathrm{ord}_E.$$
The set $X^{\mathrm{div}}$ of divisorial valuations on $X$ is dense in $X^{\mathrm{an}}$, see e.g. \cite[6.3]{jonmus}.
\end{defi}

\begin{defi}
We say that a plurisubharmonic function $\phi$ is \textbf{regularizable from below}, and we write
$$\phi\in \PSH^\uparrow$$
if and only if $\phi$ is the pointwise limit on $X^{\mathrm{div}}$ of an increasing net of Fubini-Study potentials, equivalently of an increasing net of continuous, psh functions.
\end{defi}

\begin{remark}We then have that $\phi$ is the usc regularized supremum of such a net. Furthermore, by \cite[L4.4]{ber} (which extends to our non-trivially valued case where continuity of envelopes holds), $\phi\in\PSH^\uparrow$ if and only if $\phi=Q^*(\phi)$, where $Q^*(\phi)$ denotes the usc-regularized envelope
$$Q^*(\phi)=\mathrm{usc}\,\sup\{\phi'\in \PSH(X)\cap C^0(X),\,\phi'\leq \phi\}.$$
\end{remark}

\bigskip\noindent We then show that $\PSH^\uparrow$ coincides with the image of the asymptotic Fubini-Study operator. 

\begin{theorem}
A function belongs to $\PSH^\uparrow$ if and only if there exists a bounded graded norm $\norm_\bullet$ such that $\FS(\norm_\bullet)=\phi$.
\end{theorem}
\begin{proof}
Assume $\phi$ is the image of some bounded graded norm $\norm_\bullet$ by the asymptotic Fubini-Study operator, i.e. $\phi=\mathrm{usc}\left(\lim_m \FS_m(\norm_m)\right)$. In particular, $\phi$ is psh, by continuity of envelopes. By the remark above, it is then enough to show that $\phi=\mathrm{usc}\,Q(\phi)$, which is clear by construction.

\bsni We now assume $\phi\in\PSH^\uparrow$. Now, by \cite[T7.27]{boueri}, we have that $Q(\phi)=\lim_m \FS_m(\bigN_m(\phi))$. Then, by definition, $Q^*(\phi)=\FS(\bigN_\bullet(\phi))$. Since $\phi\in\PSH^\uparrow$, we have that $\phi=Q^*(\phi)$, thus $$\phi=\FS(\bigN_\bullet(\phi)),$$ which proves the Theorem. 
\end{proof}

\subsection{The asymptotic Fubini-Study operator descends to a bijection.}

In this section, we prove the following theorem, which is a generalization of \cite[T4.16]{ber}:
\begin{theorem}[Theorem A]\label{theoremb}
Let $(X,L)$ admit continuity of envelopes, with $L$ ample. The asymptotic Fubini-Study operator $\FS$ then defines a bijection
$$\mathcal{N}_{\bullet}(L)/\sim\,\to \mathrm{PSH}^\uparrow(L).$$
\end{theorem}

\noindent The proof follows that of the aforementioned theorem. We start with preparatory lemmas:

\begin{lemma}\label{minitheoremb}Assume that $\norm_\bullet\geq \norm'_\bullet$ pointwise. Then, 
$$d_1(\norm_\bullet,\norm'_\bullet)=0\Leftrightarrow \FS(\norm_\bullet)=\FS(\norm'_\bullet).$$
\end{lemma}
\begin{proof}
We first notice that, since $\norm_\bullet\leq\norm'_\bullet$ pointwise, the definition of $d_1$ using successive minima implies
$$d_1(\norm_\bullet,\norm'_\bullet)=\vol(\norm_\bullet,\norm'_\bullet),$$
and this volume is equal to $0$ by our hypothesis. Using Theorem \ref{theorema}, we then have that
$$E(\FS(\norm_\bullet),\FS(\norm'_\bullet))=0.$$
But \cite[C6.28]{boulb} implies that $\FS(\norm'_\bullet)-\FS(\norm_\bullet)$ is constant, hence equal to zero since $E(\FS(\norm_\bullet),\FS(\norm'_\bullet))=0$.
\end{proof}

\begin{remark}Note that \cite[C6.28]{boulb} holds modulo a certain hypothesis (see \cite[R6.27]{boulb}), which is true if $(X,L)$ has property $\cpe$.
\end{remark}

\begin{lemma}\label{comparisonoperators}Let $\norm_\bullet$ be an element of $\mathcal{N}_\bullet(L)$. We then have that $\norm_\bullet\geq \bigN_\bullet(\FS(\norm_\bullet))$, and furthermore those norms are equivalent.
\end{lemma}
\begin{proof}
The first assertion follows from \cite[L7.24]{boueri} (and its proof). To show asymptotic equivalence, by the previous lemma, it is thus enough to show that $\FS(\bigN_\bullet(\FS(\norm_\bullet)))=\FS(\norm_\bullet)$. But, by \cite[T7.27]{boueri}, 
$$\FS(\bigN_\bullet(\FS(\norm_\bullet)))=Q(\FS(\norm_\bullet)),$$
which in turn is equal to $\FS(\norm_\bullet)$ itself, since it is a limit of an increasing net of Fubini-Study potentials.
\end{proof}

\noindent We now prove Theorem \ref{theoremb}.
\begin{proof}
Note that
$$d_1(\norm_\bullet,\norm'_\bullet)=\vol(\norm_\bullet,\norm_\bullet\vee\norm'_\bullet)+\vol(\norm'_\bullet,\norm_\bullet\vee\norm'_\bullet),$$
and by Theorem \ref{theorema}, the right-hand side is in fact equal to
\begin{equation}\label{distenergy1}E(\FS(\norm_\bullet),\FS(\norm_\bullet\vee\norm'_\bullet))+E(\FS(\norm'_\bullet),\FS(\norm_\bullet\vee\norm'_\bullet)).
\end{equation}
The trick is now to prove the following: 
$$\FS(\norm_\bullet\vee\norm'_\bullet)=Q(\FS(\norm_\bullet)\wedge \FS(\norm'_\bullet)),$$
where $\wedge$ denotes the $\min$ operator.
Since, by Lemma \ref{comparisonoperators},
$$\norm_\bullet\geq \bigN_\bullet(\FS(\norm_\bullet)),$$
$$\norm_\bullet\sim \bigN_\bullet(\FS(\norm_\bullet)),$$
and the same holds for $\norm'_\bullet$,
then
$$\norm_\bullet\vee\norm'_\bullet\geq \bigN_\bullet(\FS(\norm_\bullet))\wedge \bigN_\bullet(\FS(\norm'_\bullet)),$$
$$\norm_\bullet\vee\norm'_\bullet\sim \bigN_\bullet(\FS(\norm_\bullet))\wedge \bigN_\bullet(\FS(\norm'_\bullet)),$$
and furthermore, by \cite[(4.4)]{ber}, 
$$\bigN_\bullet(\FS(\norm_\bullet))\vee \bigN_\bullet(\FS(\norm'_\bullet))=\bigN_\bullet(\FS(\norm_\bullet)\wedge \FS(\norm'_\bullet)).$$
Lemma \ref{minitheoremb} then implies
$$\FS(\bigN_\bullet(\FS(\norm_\bullet)\wedge \FS(\norm'_\bullet)))=\FS(\norm_\bullet\vee\norm'_\bullet),$$
and the left-hand side is equal to $Q(\FS(\norm_\bullet)\wedge \FS(\norm'_\bullet))$, by \cite[T7.27]{boueri}. We may now rewrite \eqref{distenergy1} as:
$$d_1(\norm_\bullet,\norm'_\bullet)=E(\FS(\norm_\bullet),Q(\FS(\norm_\bullet)\wedge \FS(\norm'_\bullet)))+E(\FS(\norm'_\bullet),Q(\FS(\norm_\bullet)\wedge \FS(\norm'_\bullet))).$$
Now, 
$$\FS(\norm_\bullet)\geq Q(\FS(\norm_\bullet)\wedge \FS(\norm'_\bullet))$$
and
$$\FS(\norm'_\bullet)\geq Q(\FS(\norm_\bullet)\wedge \FS(\norm'_\bullet)),$$
so that the two energies above have the same sign. In particular, the distance $d_1(\norm_\bullet,\norm'_\bullet)$ vanishes if and only if the energies vanish, and we conclude using Lemma \ref{minitheoremb}.
\end{proof}

\begin{remark}Note that the previous proof shows that there is an expression of the $d_1$ distance using Monge-Ampère energies, analogous to \cite[C4.21]{ber}:
$$d_1(\norm_\bullet,\norm'_\bullet)=E(\FS(\norm_\bullet),Q(\FS(\norm_\bullet)\wedge \FS(\norm'_\bullet)))+E(\FS(\norm'_\bullet),Q(\FS(\norm_\bullet)\wedge \FS(\norm'_\bullet))).$$
\end{remark}

\newpage

\bibliographystyle{alpha}
\bibliography{bib}

\end{document}